\documentclass[11pt, a4paper,leqno]{amsart}
\usepackage{amsmath,amsthm,amscd,amssymb,amsfonts, amsbsy}
\usepackage{latexsym}
\usepackage{txfonts}
\usepackage{exscale}
\usepackage{pdfsync}

\usepackage{bbm}
\usepackage[colorlinks,citecolor=red,pagebackref,hypertexnames=false]{hyperref}
\usepackage{pdfsync}
\usepackage{color}

\calclayout
\allowdisplaybreaks

\theoremstyle{plain}
\newtheorem{theorem}[equation]{Theorem}
\newtheorem{lemma}[equation]{Lemma}
\newtheorem{corollary}[equation]{Corollary}
\newtheorem{proposition}[equation]{Proposition}

\theoremstyle{definition}
\newtheorem{definition}[equation]{Definition}

\theoremstyle{remark}
\newtheorem{remark}[equation]{Remark}

\numberwithin{equation}{section}

\newcommand{\RR}{{\mathbb{R}}}

\newcommand{\CC}{{\mathbb{C}}}
\newcommand{\eps}{\varepsilon}

\newcommand{\N}{\mathbb{N}}

\newcommand{\Z}{\mathbb{Z}}

\newcommand{\om}{\Omega}

\newcommand{\G}{\mathcal{G}}

\newcommand{\RNum}[1]{\uppercase\expandafter{\romannumeral #1\relax}}

\newcommand{\norm}[1]{\left \lVert #1 \right \rVert }

\newcommand{\vw}{$A_{\infty}^{vw}$\ }
\newcommand{\vwp}{$\mathcal{A}_{\infty}^{vw}$\ }
\newcommand{\rhq}{$\mathcal{RH}^{vw}$\ }

\newcommand{\cl}[1]{\overline{#1}}

\author{Pascal Auscher}
\author{Simon Bortz}
\author{Moritz Egert}
\author{Olli Saari}
\address{Laboratoire de Math\'{e}matiques d'Orsay, Univ. Paris-Sud, CNRS, Universit\'{e} Paris-Saclay, 91405 Orsay, France \vspace{5pt}\newline\vspace{5pt}{\rm and}\newline Laboratoire Ami\'{e}nois de Math\'{e}matique Fondamentale et Appliqu\'{e}e, UMR 7352 du CNRS, Universit\'{e} de Picardie-Jules Verne, 80039 Amiens, France}
\email{pascal.auscher@math.u-psud.fr}

\address{Laboratoire de Math\'{e}matiques d'Orsay, Univ. Paris-Sud, CNRS, Universit\'{e} Paris-Saclay, 91405 Orsay, France}
\email{moritz.egert@math.u-psud.fr}

\address{School of Mathematics, University of Minnesota, Minneapolis, MN 55455, USA}
\email{bortz010@umn.edu} 

\address{Department of Mathematics and Systems Analysis, Aalto University, FI-00076 Aalto, Finland \vspace{5pt}\newline\vspace{5pt}{\rm and}\newline Mathematical Institute, University of Bonn, 53115 Bonn, Germany
}
\email{saari@math.uni-bonn.de}

\thanks{The first and third authors were partially supported by the ANR project ``Harmonic Analysis at its Boundaries'', ANR-12-BS01-0013. This material is based upon work supported by National Science Foundation under Grant No.\ DMS-1440140 while the  authors were in residence at the MSRI in Berkeley, California, during the Spring 2017 semester. The second author was supported by the NSF INSPIRE Award DMS
1344235. The third author was supported by a public grant as part of the FMJH}



\begin{document}
\allowdisplaybreaks

\title[Non-local Gehring lemmas]{Non-local Gehring lemmas in spaces of homogeneous type and applications}

\begin{abstract}
We prove a self-improving property for reverse H\"older inequalities with non-local right-hand side.  We attempt to cover all the most important situations that one encounters when studying elliptic and parabolic partial differential equations. We present applications to non-local extensions of $A_{\infty}$ weights and fractional elliptic divergence form equations. We write our results in spaces of homogeneous type.
\end{abstract}

\maketitle
\setcounter{tocdepth}{1}
\tableofcontents
\setcounter{tocdepth}{2}
\date{\today}

\section{Introduction}

Gehring's lemma \cite{Gehring1973} establishes the open-ended property of reverse H\"older classes. If 
\begin{equation}
\label{intro1}
\left( \frac{1}{|B|} \int_{B} u^{q}\, dx \right)^{1/q} \lesssim \frac{1}{|B|} \int_{B} u \, dx 
\end{equation}
with $q > 1$ and all Euclidean balls $B \subset \mathbb{R}^{n}$, then 
\[ \left( \frac{1}{|B|} \int_{B} u^{q + \epsilon}\, dx \right)^{1/(q + \epsilon)} \lesssim_{q} \frac{1}{|B|} \int_{B} u \,  dx \]
for a certain $\epsilon > 0$ and all Euclidean balls. This self-improving property has proved to be an important tool when studying elliptic \cite{EM1975,Gia} and parabolic \cite{GS1982} partial differential equations as well as quasiconformal mappings \cite{IN1985}. In this case, one has to enlarge the ball in the right  hand side. We come back to this. 

In this work, we are concerned with reverse H\"older inequalities when the right-hand side is non-local. Understanding an analogue of Gehring's lemma in this generality turned out to be crucial in \cite{ABES2017}, where we prove H\"older continuity in time  for  solutions of parabolic systems. The non-local nature   arises from the use of half-order time derivatives.  The ambient space being quasi-metric instead of Euclidean is also an assumption natural from the point of view of parabolic partial differential equations. Hence, we shall explore these non-local Gehring lemmas in spaces of homogeneous type.

It is well known that Gehring's lemma holds for the so called weak reverse H\"older inequality where the right-hand side of \eqref{intro1} is an average over a dilated ball $2B$. We replace the single dilate by a {significantly weaker} non-local tail such as
\[\sum_{k= 0}^{\infty} 2^{-k} \frac{1}{ |2^{k} B|} \int _{2^{k}B}  u \, dx \]
and certain averages over additional functions $f$ and $h$ that have a special meaning in applications. The main result of this paper is Theorem \ref{GeneralGehering.quasi} asserting that a variant of Gehring's lemma, and in particular the local higher integrability of $u$ still holds in this setting. 
We present a core version of the theorem already in the next section. It comes with the introduction of some necessary notation but we tried to keep things simple to give the reader a first flavor of our results. Once the strategy is in place,
we discuss various consequences (Section~\ref{sec:global}), ways to generalize it (Sections~\ref{sec:variants} and \ref{sec:extensions}) as well as self-improving properties for the right-hand side of the reverse H\"older inequality with tail (Section \ref{sec:sirhs}). We aim at covering all the aspects that usually arise from applications. 
Finally, we illustrate our main result by an application to regularity of solutions to the fractional divergence form equation introduced in \cite{ShiehFractional} and investigated further for example in \cite{SchikorraVMO}. 

The context of our work is the following. Gehring's lemma in a metric space endowed with a doubling measure was proved in \cite{ZG2005}. See also the book \cite{BB2011}. By \cite{MS1979}, every quasi-metric space carries a compatible metric structure so that Gehring's lemma also holds in that setting. However, in the case of homogeneous reverse H\"older inequalities, a very clean argument using self-improving properties of $A_\infty$ weights was used in \cite{AHT2017} to give an intrinsically quasi-metric proof (see also the very closely related work \cite{HPR2012}). We do not attempt to review the literature in the Euclidean $n$-space, but we refer to the excellent survey in \cite{Iwaniec1995} instead. In addition, we want to point out the recent paper on Gehring's lemma for fractional Sobolev spaces \cite{KMS2015}. That paper studies fractional equations, whose solutions are self-improving in terms of both integrability and differentiability. Such phenomena are different from what we encounter here, but we found the technical part of \cite{KMS2015} very inspiring.  

Among generalizations, we mention that the tails may be replaced by some supremum of averages taken over balls larger than the original ball on the left-hand side and/or that one may work on open subsets. In this way, our methods can also be applied to obtain a generalization of $A_{\infty}$ weights: In \cite{AHT2017},  a larger class of \emph{weak $A_{\infty}$ weights}, generalizing the one considered  in \cite{Fujii, W1987}  was defined  and their higher (than one) integrability was  proved  (in spaces of homogeneous type). This class of weights, larger than the usual $A_{\infty}$ Muckenhoupt class, is defined  by allowing a uniform  dilation of the ball in the right-hand side  compared to the one on the left-hand side. Here, we show that, in fact, the dilation  may be  arbitrary (depending on the ball) provided it is finite. Another family of weights covered by our methods is the $C_p$ class studied in \cite{Muckenhoupt1981,Sawyer1983}. Precise definitions are given in Section~\ref{sec:Ainfty}.

\subsection*{Acknowledgment}
We thank Tuomas Hyt\"onen for an enlightening discussion on the topics of this work that led to the results extending the $A_{\infty}$ class and Carlos P\'erez for pointing out the connection to the $C_p$ class. We also thank an anonymous referee for suggesting that our results should apply to the fractional divergence form equation of Shieh--Spector \cite{ShiehFractional} rather than the toy model investigated in an earlier version of our manuscript.

\section{Metric spaces}
A \emph{space of homogeneous type} $(X,d,\mu)$ is a triple consisting of a set $X$, a function $d : X \times X \to [0,\infty) $ satisfying the quasi-distance axioms
\begin{enumerate}
\item[(i)] $d(x,y) = 0$ if and only if $x = y$,
\item[(ii)] $d(x,y) = d(y,x)$ for all $x,y \in X$, and 
\item[(iii)] $d(x,z) \leq K ( d(x,y) + d(y,z))$ for a certain $K \geq 1$ and all $x,y,z \in X$;
\end{enumerate} 
and a Borel measure $\mu$ that is doubling in the sense that 
\[0 < \mu ( B(x,2r) ) \leq C_d \, \mu ( B(x,r) ) < \infty  \]
holds for a certain $C_d$ and all radii $r > 0$ and centers $x \in X$. If the constant $K$ appearing in the triangle inequality (iii) equals $1$, we call $(X,d,\mu)$ a \emph{metric space} with doubling measure. The topology is understood to be the one generated by the quasi-metric balls. For simplicity, we impose the additional assumption that all quasi-metric balls are Borel measurable. In general, they can even fail to be open.  

The doubling condition implies there is $C > 0$ so that for some $D > 0$,
\begin{equation}
\label{eq:homdim}
\frac{\mu(B(x,R))}{\mu(B(x,r))} \leq C \left( \frac{R}{r} \right)^{D}
\end{equation}
for all $x \in X$ and $R \geq r > 0$. We can always take $D = \log_2 C_d$. In the following we call this number the \emph{homogeneous dimension} (although there might be smaller positive numbers $D$ than $\log_2 C_d$ for which  this inequality holds: our proofs work with any such $D$). For all these basic facts on analysis in metric spaces, we refer to the book \cite{BB2011}. 

The following theorem is concerned with the special case of metric spaces, but it has an analogue in the general case of quasi-metric spaces, see Theorem \ref{GeneralGehering.quasi} below. 

\begin{theorem}\label{GeneralGehering.cor}
Let $(X,d,\mu)$ be a metric space with doubling measure. Let {$s, \beta >0$ and $q>1$ be such that $s<q$} 
and $\beta \geq D(1/s-1/q)$ where $D$ is any number satisfying \eqref{eq:homdim}. Let $N >1$ and let $(\alpha_k)_{k \geq 0}$ be a non-increasing sequence of positive numbers with $\alpha := \sum_k \alpha_k  < \infty$, and define 
\begin{equation}
\label{eq:aub}
a_u (B) := \sum_{k=0}^{\infty} \alpha_k \fint_{N^{k} B} u \, d \mu
\end{equation}  
for $u \geq 0$ locally integrable and $B$ a metric ball. 

Suppose that $u, f, h \ge 0$ with $u^q, f^q, h^s \in L^1_{loc}(X, d\mu)$ and  $A\ge 0$ is a constant  such that for every ball $B = B(x, R)$, 
\begin{equation}\label{GGeq1cor.eq}
\left(\fint_B u^q \, d\mu\right)^{1/q} \le Aa_{u}(B) + (a_{f^q}(B))^{1/q} + R^\beta(a_{h^s}(B))^{1/s}.
\end{equation} 
Then there exists $p > q$ depending on $\alpha_0, \alpha, A, q, s,N$ and $C_d$ such that for all balls $B$, 
\begin{equation}\label{GGeq2cor.eq}
\begin{split}
\left(\fint_B u^p \, d\mu\right)^{1/p} &\lesssim a_{u}(NB) + (a_{f^q}(NB))^{1/q} + R^\beta(a_{h^s}(NB))^{1/s} 
\\& \qquad+ \left(\fint_{N B}f^p \, d\mu\right)^{1/p} + R^\beta\left(\fint_{N B}h^{ps/q} \, d\mu\right)^{q/sp},
\end{split}
\end{equation}
with implicit constant depending on $\alpha_0,\alpha,  A, q, s, \beta,N$ and  $C_d$.

\begin{remark} If one assumes the  sequence $(\alpha_{k})_{k\ge 0}$ is finite, {the functional is comparable to one single average on $N^{k_{0}}B$ for some $k_{0}$. This gives a proof of the classical Gehring lemma with dilated balls. Note the shift from $N^{k_{0}}B$ to $N^{k_{0}+1}B$ in the conclusion.  But well-known additional covering arguments show that the dilation factor $N^{k_{0}+1}$ can be changed to any number larger than $1$. }
If one assumes 
\begin{equation}
\label{eq:C}
\exists  C<\infty : \forall k\ge 0 \quad \alpha_{k} \le C \alpha_{k+1},
\end{equation} then it follows that $a_{u}(NB) \le C a_{u}(B)$ for all $u\ge 0 $ and all balls $B$. In that case, one can replace $NB$ by $B$ in the right-hand side of \eqref{GGeq2cor.eq}. Geometric sequences, which are typical in application,  do satisfy this condition but this rules out finite sequences.  Finally, note that the higher integrability of $u$ on $B$ depends only on the higher integrability of $f$ and $h$ on the first dilated ball $NB$.  
\end{remark}

\end{theorem}

\begin{proof} 
We prove \eqref{GGeq2cor.eq} for $B = B(x_0,R)$ with $x_{0} \in X$ and $R> 0$. {Throughout, we reserve the symbol $C$ for a constant that depends at most on $\alpha_0,\alpha,  A, q, s, \beta,N$ and  $C_d$ but that may vary from line to line.}

\subsection*{Step 1. Preparation}  

Having fixed $B$,  we set 
$g^q := A_{R}^{q} h^s \mathbbm{1}_{NB}$ with $A_{R}$ a constant so that for any ball $B_{r}$ with radius $r$ contained in $NB$,
we have 
\begin{equation}
\label{eq:trick}
r^\beta \left(\fint_{B_{r}}h^s \, d\mu\right)^{1/s}  \le    \left(\fint_{B_{r}}g^q \, d\mu\right)^{1/q}
\end{equation} 
and
\begin{equation}
\label{eq:trick2}
    \left(\fint_{ N B }g^q \, d\mu\right)^{1/q} \le C_{1} (N R)^\beta\left(\fint_{N B}h^s \, d\mu\right)^{1/s}  
\end{equation} 
for some $C_{1}$ depending only on the doubling condition, $s$ and $q$. Indeed, write $B_{r}=B(x,r)$. As $x\in NB$, we have $NB=B(x_{0},N R ) \subset B(x, 2 N R)$, hence 
$$
\frac{\mu(NB)}{\mu(B_{r})} \le \frac{\mu(B(x,2N R))}{\mu(B(x,r))} \le C_{0} \left(\frac{2NR}{r}\right)^D \le C_{0} 2^D \left(\frac{NR}{r}\right)^{\beta(1/s-1/q)^{-1}}
$$
where $C_{0}$ depends only on the doubling condition.  Unraveling this inequality and setting $C_{1}  =(C_{0}2^D)^{1/q-1/s}$ yield
$$
r^\beta \mu(B_{r})^{1/q-1/s} \le C_{1} (N R)^\beta \mu(N B)^{1/q-1/s}.
$$
Hence, as $q>s$, 
$$
r^\beta \mu(B_{r})^{1/q-1/s} \left(\int_{B_{r}}h^s \, d\mu\right)^{1/s-1/q} \le C_{1}(NR)^\beta \mu(NB)^{1/q-1/s} \left(\int_{NB}h^s \, d\mu\right)^{1/s-1/q} 
$$
so that
\begin{align*}
   r^\beta\left(\fint_{B_{r}}h^s \, d\mu\right)^{1/s} &= r^\beta\left(\fint_{B_{r}}h^s \, d\mu\right)^{1/s-1/q}\left(\fint_{B_{r}}h^s \, d\mu\right)^{1/q} \\
   &\le C_{1}(NR)^\beta \left(\fint_{NB}h^s \, d\mu\right)^{1/s-1/q}   \left(\fint_{B_{r}}h^s \, d\mu\right)^{1/q}.  
\end{align*}
Thus, we set 
\begin{align}
\label{eq:AR}
A_{R}:=C_{1}(NR)^\beta \left(\fint_{NB}h^s \, d\mu\right)^{1/s-1/q}
\end{align}
and \eqref{eq:trick} is proved.
Observing that if $B_{r}=NB$ we have equalities with constant $1$ in the inequalities above, the constant $C_{1}$ works for \eqref{eq:trick2}. 
 
\subsection*{Step 2. Local setup}

For $\ell \in \N$, fix $r_0$ and $\rho_0$ real numbers satisfying $R \le r_0 <  \rho_0 \le NR$ with $N^\ell(\rho_0 - r_0) = R$. For $x \in B(x_0, r_0)$, we have that $N^kN^\ell (\rho_0 - r_0) = N^kR$ for $k \ge 0$ so
$$B(x, N^{k}(\rho_0 - r_0)) \subset B(x_0, N^{k+1} R) \subset B(x, N^{k+ \ell + j}(\rho_0 - r_0)),$$
where  $j=2$ when $N\ge 2$  and $j= \lceil (\log_{2}N)^{-1}+1\rceil$ when $1<N<2$, and consequently for any positive $\mu$-measurable function $v$,
\begin{equation}
\label{doubavgineqcor.eq}
\begin{split}
\fint_{B(x,N^k(\rho_0 - r_0))} v \, d\mu & = \frac{1}{\mu(B(x,N^k(\rho_0 - r_0)))} \int_{B(x,N^k(\rho_0 - r_0))} v \, d\mu
\\& \le  \frac{\mu(B(x, N^{k+ \ell + j}(\rho_0 - r_0)))}{\mu(B(x,N^k(\rho_0 - r_0)))} \frac{1}{B(x_0,N^{k+1}R)} \int_{B(x_0,N^{k+1}R)} v \, d\mu
\\& \le C_N^{\ell +j}  \fint_{B(x_0,N^{k+1}R)} v \, d\mu,
\end{split}
\end{equation}
where we used \eqref{eq:homdim} in the last line. The constant $C_N \geq 1$ only depends on the doubling constant $C_d$ and the number $N$. 

\subsection*{Step 3. Beginning of the estimate}  

For $m >0$, set $u_m: = \min\{u,m\}$, $B_{r_0} := B(x_0,r_0)$ and $B_{\rho_0} := B(x_0,\rho_0)$. Using the Lebesgue-Stieltjes formulation of the integral we have
\begin{equation}
\label{eqstartlayercake.eq}
\begin{split}
\int_{B_{r_0}} u_m^{p-q}u^q \, d\mu & = (p-q)\int_0^m \lambda^{p -q - 1} u^q(B_{r_0} \cap \{u > \lambda\})\, d\lambda
\\&= (p-q)\int_0^{\lambda_0} \lambda^{p -q - 1} u^q(B_{r_0} \cap \{u > \lambda\})\, d\lambda
\\ &\qquad + (p-q)\int_{\lambda_0}^m \lambda^{p -q - 1} u^q(B_{r_0} \cap \{u > \lambda\})\, d\lambda
\\& \le \lambda_0^{p-q} u^q(B_{r_0})  + (p-q)\int_{\lambda_0}^m \lambda^{p -q - 1} u^q(B_{r_0} \cap \{u > \lambda\})\, d\lambda
\\& =: I + II,
\end{split}
\end{equation}
where $u^q(\mathcal{A}) = \int_{\mathcal{A}} u^q \, d \mu$ for any measurable set $\mathcal{A} \subseteq X$ and $\lambda_{0}$ is a constant chosen below. 

\subsection*{Step 4. Choice of the threshold $\lambda_{0}$}   
We define three functions 
\begin{equation*}
U(x,r): = \fint_{B(x,r)} u \, d \mu, \ \ \  F(x,r): = \left(\fint_{B(x,r)} f^q \, d \mu\right)^{1/q},  
\ \ \ G(x,r): = \left(\fint_{B(x,r)} g^q \, d \mu\right)^{1/q}
\end{equation*}
and for $\lambda > \lambda_0$, we denote the relevant level sets by
\begin{align*}
U_\lambda := B_{r_0} \cap \{u > \lambda\}, \qquad
F_\lambda := B_{r_0} \cap \{f > \lambda\}, \qquad
G_\lambda := B_{r_0} \cap \{g > \lambda\}.
\end{align*}
 It follows from \eqref{doubavgineqcor.eq} with $k = 0$ that for $x \in B_{r_0}$,
$$
U(x,\rho_0 - r_0) = \fint_{B(x,\rho_0 - r_0)} u \, d \mu \le \frac{C_N^{\ell + j}}{\alpha_0}a_u(NB)$$
and one has the same observation for $F$ with $f^q$. For the last term, we use \eqref{doubavgineqcor.eq} in conjunction with \eqref{eq:trick2} to obtain
\begin{align*}
 \fint_{B(x,\rho_0 - r_0)}  g^q\, d \mu &\le {C_{N}^{\ell + j}}\fint_{NB}  g^q\, d \mu     \le {C_{N}^{\ell + j}} C_{1}^q (NR)^{\beta q} \left(\fint_{NB}h^s \, d\mu\right)^{q/s} \\
 & \le  \frac{C_N^{(\ell + j)q/s}}{\alpha_0^{q/s}} C_{1}^q  (NR)^{\beta q} a_{h^s}(NB)^{q/s}, 
\end{align*}
where we used $C_N \ge 1$ and $q/s>1$.  Consequently, we choose 
\begin{equation*}
\lambda_0:= \frac{C_N^{\ell + j}}{\alpha_0}a_u(N B) + \left(  \frac{C_N^{\ell + j}}{\alpha_0}a_{f^q}(NB)\right)^{1/q}  +  C_{1}(NR)^\beta  \left(  \frac{C_N^{\ell + j}}{\alpha_0} a_{h^s}(NB)\right)^{1/s}.
\end{equation*} 
Finally, set $$\Omega_\lambda : = \Big \{x \in U_\lambda \cup F_\lambda \cup G_\lambda: \text{$x$ is a Lebesgue point for $u, f^q$ and $g^q$} \Big \}.$$

\subsection*{Step 5. Estimate of the measure of $U_{\lambda}$}

We begin to estimate $II$ in \eqref{eqstartlayercake.eq} so we assume $\lambda>\lambda_{0}$.
For  $x \in B_{r_0}$,  
\begin{equation}
\label{startcor.eq}
U(x,\rho_0 - r_0) + F(x,\rho_0 - r_0) + G(x,\rho_0 - r_0) \le \lambda_0 < \lambda.
\end{equation}
On the other hand, by definition of $U_\lambda$, $F_\lambda$ and $G_\lambda$, if $x \in \Omega_\lambda$ then
$$\lim_{r \to 0} U(x,r) + F(x,r) + G(x,r) > \lambda$$
Thus for $x \in \Omega_\lambda$ we can define the stopping time radius
\begin{equation*}
\begin{split}
r_x&:= \sup \Big\{ N^{-m}(\rho_0 - r_0) : m \in \N 
\\ & \quad \quad \text{ and } U(x,N^{-m}(\rho_0 - r_0)) + F(x,N^{-m}(\rho_0 - r_0)) + G(x,N^{-m}(\rho_0 - r_0)) > \lambda \Big\}.
\end{split}
\end{equation*}
We remark that \eqref{startcor.eq} implies that $r_{x}< \rho_0 - r_0$. 
Of course $\Omega_\lambda \subset \cup_{x \in \Omega_\lambda} B(x, r_x/5)$.  
By the Vitali Covering Lemma (5r-Covering Lemma) there exists a countable collection of balls $\{B(x_i,r_i)\} = \{B_i\}$ with $r_i = r_{x_i}$ such that $\{\tfrac{1}{5}B_i\}$ are pairwise disjoint and $\Omega_\lambda \subset \cup_i B_i$.  Let $m_{i}\ge 1$ such that $N^{m_{i}}r_{i}=\rho_0 - r_0$. 

We make three observations:
\begin{enumerate}
\item[(i)] For each $i$, either $\fint_{B_i} u \, d \mu > \frac \lambda 3$, $(\fint_{B_i} f^q \, d \mu)^{1/q} > \frac \lambda 3$, or $(\fint_{B_i} g^q \, d \mu)^{1/q} > \frac  \lambda 3$.
\item[(ii)] The radius of each $B_i$ is less than $\rho_0 - r_0$ and $x_i \in B_{r_0}$ so $B_i \subset B_{\rho_0}$.
\item[(iii)] For $0  \le k <  m_{i}$,  $N^k r_i= N^{-(m_{i}-k)}( \rho_0 - r_0) < \rho_0 - r_0$, so $N^{k} r_i$ is `above' or at the stopping time and 
$$\fint_{N^k B_i} u \, d\mu + \left(\fint_{N^k B_i} f^q \, d\mu \right)^{1/q} +  \left(\fint_{N^k B_i} g^q \, d\mu \right)^{1/q} \le {C} \lambda,$$
where $C$ { shows up since we have used doubling once in the case $k=0$}.
\end{enumerate}

Using that $\mu((U_\lambda \cup F_\lambda \cup G_\lambda)\setminus \Omega_\lambda) = 0$, $\Omega_\lambda \subset \cup_i B_i$ and \eqref{GGeq1cor.eq}, we obtain
\begin{equation}\label{uqofelamcor.eq}
\begin{split}
u^q(U_\lambda) &\le u^{q}(U_\lambda \cup F_\lambda \cup G_\lambda) \le \sum_i  u^q(B_{i})
\\& \le \sum_i \mu(B_i)[Aa_u(B_i) + (a_{f^q}(B_i))^{1/q} + r_{i}^\beta (a_{h^s}(B_i))^{1/s}]^q.
\end{split}
\end{equation} 
Then using  $m_i \ge 1$, $\sum_k \alpha_k = \alpha$ and observation (iii) we obtain 
\begin{equation}
\label{eq:splittingthesum}
\begin{split}
a_{f^q}(B_i) &= \sum_{k=0}^\infty \alpha_k \fint_{N^k B_i} f^q \, d\mu = \sum_{k=0}^{m_i - 1} \alpha_k \fint_{N^k B_i} f^q \, d\mu + \sum_{k=m_i}^\infty \alpha_k \fint_{N^k B_i} f^q \, d\mu 
\\ &\le { C^{q}} \alpha \lambda^{q} + \sum_{k=0}^\infty \alpha_{k + m_i} \fint_{B(x_i, N^k(\rho_0 - r_0))} f^q \, d\mu,
\end{split}
\end{equation}
where we simply re-indexed the second sum and used that $N^{m_i}r_i = \rho_0 - r_0$.  Now we use \eqref{doubavgineqcor.eq} and that $\alpha_{k + m_i} \le \alpha_{k}$  to deduce
\begin{equation*}\label{auboundonBicor_f.eq}
\begin{split}
a_{f^q}(B_i) &\le { C^{q}} \alpha \lambda^{q} + C_{N}^{\ell +j} \sum_{k=0}^\infty \alpha_{k} \fint_{N^{k+1}B} f^q \, d\mu
\\& \le { C^{q}}  \alpha\lambda ^{q} + C_{N}^{\ell + j} a_{f^q}(NB) \le C^q \alpha\lambda^{q} + \alpha_0 \lambda_0^{q} 
\\& \leq  C \lambda^{q},
\end{split}
\end{equation*}
where we used the definition of $\lambda_0$ in Step~4 and $\lambda > \lambda_0$. Similarly,  $a_{u}(B_i) \leq C \lambda$.
For $h^s$,  using  $m_i \ge 1$,  we obtain 
\begin{equation*}
\begin{split}
(r_{i})^{\beta s}& a_{h^s}(B_i) \\& = \sum_{k=0}^\infty \alpha_k (r_{i})^{\beta s} \fint_{N^k B_i} h^s \, d\mu 
\\
&= \sum_{k=0}^{m_i - 1} \alpha_k (r_{i})^{\beta s} \fint_{N^k B_i} h^s \, d\mu + \sum_{k=m_i}^\infty \alpha_k (r_{i})^{\beta s} \fint_{N^k B_i} h^s \, d\mu 
\\ &\le \sum_{k=0}^{m_i - 1} \alpha_k \left( \fint_{N^k B_i} g^q \, d\mu\right)^{s/q}   + (r_{i})^{\beta s} \sum_{k=0}^\infty \alpha_{k + m_i}   \fint_{B(x_i, N^k(\rho_0 - r_0))} h^s \, d\mu,
\end{split}
\end{equation*}
where we used \eqref{eq:trick} and $N^kB_{i}\subset NB$ when $k<m_{i}$ for the first sum, re-indexed the second sum and used that $N^{m_i}r_i = \rho_0 - r_0$.  With  $\sum_k \alpha_k = \alpha$ and observation (iii) for the first sum and $\alpha_{k + m_i} \le \alpha_{k}$ along with \eqref{doubavgineqcor.eq} for the second one, we deduce
\begin{equation*}\label{auboundonBicor_h.eq}
\begin{split}
(r_{i})^{\beta s} a_{h^{s}}(B_i) &\le { C^{s}} \alpha \lambda^{s} + C_{N}^{\ell +j} (NR)^{\beta  s} \sum_{k=0}^\infty {\alpha_{k}}  \fint_{N^{k+1}B} h^s \, d\mu
\\& \le { C^{s}}  \alpha\lambda ^{s} + C_{N}^{\ell + j}   (NR)^{\beta s} a_{h^{s}}(NB) \le C^s \alpha\lambda^{s} + \alpha_0 \lambda_0^{s} C_{1}^{-s}
\\& \leq C \lambda^{s},
\end{split}
\end{equation*}
where we used $\lambda > \lambda_0$. Combining the above estimates  with \eqref{uqofelamcor.eq} we obtain 
\begin{equation}
\label{uqbnd1cor.eq}
u^q(U_\lambda) \le C \lambda^q \sum_i \mu(B_i) 
\le CC_{d}^3\lambda^q \sum_i \mu\left(\tfrac{1}{5}B_i\right) 
\le CC_{d}^3 \lambda^q  \mu\left(\cup_i  B_i\right)
\end{equation}
where we used that $\{\tfrac{1}{5}B_i\}$ are pairwise disjoint.  Now we use (i) and (ii) to conclude that
\begin{equation}
\label{uqbnd2cor.eq}
\cup_i  B_i \subset  \{ M(u\mathbbm{1}_{B_{\rho_0}}) > \lambda/3\}\cup  \{  M(f^q\mathbbm{1}_{B_{\rho_0}}) > (\lambda/3)^q\}
 \cup \{ M(g^q\mathbbm{1}_{B_{\rho_0}}) > (\lambda/3)^q\},
\end{equation}
where $M$ is the uncentered maximal function. 

\subsection*{Step 6. Estimate of $II$ and $I$}

Plugging \eqref{uqbnd1cor.eq} and  \eqref{uqbnd2cor.eq} into $II$ we obtain
\begin{equation*}
\begin{split}
II &= (p-q)\int_{\lambda_0}^m \lambda^{p -q - 1} u^q(U_\lambda)\, d\lambda
\\& \le C(p-q)\int_0^m \lambda^{p-1} \mu(\{ M(u\mathbbm{1}_{B_{\rho_0}}) > \lambda/3\}) \, d\lambda
\\& \quad  + C(p-q)\int_0^m \lambda^{p-1} \mu(\{M(f^q\mathbbm{1}_{B_{\rho_0}}) > (\lambda/3)^q\})\, d\lambda
\\& \quad   +C(p-q)\int_0^m \lambda^{p-1} \mu(\{ M(g^q\mathbbm{1}_{B_{\rho_0}}) > (\lambda/3)^q\}) \, d\lambda
\\& =: II_1 + II_2 + II_3.
\end{split}
\end{equation*}

We handle $II_2$ and $II_3$ in the same way. Using the Hardy-Littlewood maximal theorem for spaces of homogeneous type to the effect that for $p \in (q,2q)$ the $L^{p/q} \to L^{p/q}$ operator norm of the maximal function is bounded by $C\tfrac{p/q}{(p/q) - 1}$, see Theorem~3.13 in \cite{BB2011}, we obtain 
\begin{equation*}
\begin{split}
II_3 &= C(p-q)\int_0^m \lambda^{p-1} \mu (\{M(g^q\mathbbm{1}_{B_{\rho_0}}) > (\lambda/3)^q)\} ) \, d\lambda
\\& \le C\frac{p-q}{p} \int_X (M(g^q\mathbbm{1}_{B_{\rho_0}}))^{p/q} \, d\mu
\\& \le C {\Big(\frac{p}{p-q}\Big)^{p/q -1}}  \int_{NB} g^p \, d\mu,
\end{split}
\end{equation*}
where we used $B_{\rho_{0}} \subset NB$ in the last step.
Similarly we have that $$II_2 \le C \big(\frac{p}{p-q}\big)^{p/q -1} \int_{NB} f^p\, d\mu.$$ To handle $II_1$ we notice that 
$$\{M(u\mathbbm{1}_{B_{\rho_0}}) > \lambda/3\} \subset \{M(u\mathbbm{1}_{B_{\rho_0}\cap \{u > \lambda/6\}}) > \lambda/6\}.$$
From this estimate and the weak type (1,1) bound for the Hardy-Littlewood maximal function for spaces of homogeneous type, see Lemma~3.12 in \cite{BB2011}, we have
$$\mu(\{M(u\mathbbm{1}_{B_{\rho_0}}) > \lambda/3\}) \le \frac{C}{\lambda} \int_{B_{\rho_0}\cap \{u > \lambda/6\}} u \, d\mu.$$
Using this bound in $II_1$ yields 
\begin{equation}
\begin{split}
II_1 &\le C(p-q)\int_0^m \lambda^{p-2}\int_{B_{\rho_0}\cap \{u > \lambda/6\}} u \, d\mu \, d \lambda
\\& = C(p-q) \int_{B_{\rho_0}}u \int_0^{\max\{m, 6u\}} \lambda^{p-2} \, d\lambda \, d\mu
\\& = C 6^{p-1} \frac{p-q}{p-1} \int_{B_{\rho_0}}u_{m/6}^{p-1} u \, d \mu
\\& \le C6^{p-1} \frac{p-q}{p-1} \int_{B_{\rho_0}}u_{m}^{p-q} u^{q} \, d \mu,
\end{split}
\end{equation}
and we note that we can make the constant in front of the integral arbitrarily small by choice of $p > q$. Combining our estimates for $II_1, II_2$ and $II_3$ we obtain for any {$p \in (q, 2q)$,
\begin{equation}\label{boundfortwocor.eq}
II \le \epsilon_p \int_{B_{\rho_0}}u_{m}^{p-q} u^q \, d \mu + C \epsilon_p^{-1} \int_{NB} f^p\, d\mu +C \epsilon_p^{-1} \int_{NB} g^p \, d\mu,
\end{equation}
where $\epsilon_p := C(p-q)$. }

Now we bound $I$. Note that $B\subset  B_{r_0} \subset NB$. By definition of $\lambda_0$ and using \eqref{GGeq1cor.eq},
\begin{equation}\label{boundforonecor.eq}
\begin{split}
I & \leq \lambda_0^{p-q} u^q(NB) 
\\& \le \lambda_0^{p-q} \mu(NB)\big( Aa_u(NB) + (a_{f^q}(NB))^{1/q} + (NR)^{\beta }(a_{h^s}(NB))^{1/s}\big)^q
\\& \le C (\widetilde{C}_N^{\ell+j})^{p-q} \mu(NB) \alpha_p(NB),
\end{split}
\end{equation}
where we denoted $\alpha_p(NB) := \big(a_u(NB) + (a_{f^q}(NB))^{1/q} + (NR)^{\beta }(a_{h^s}(NB))^{1/s}\big)^p$ {and put $\widetilde{C}_N := \max(C_N, C_N^{1/s}) \geq 1$ on recalling that we allow for $s < 1$}.

\subsection*{Step 7. Conclusion}

Setting 
\[\varphi(t) := \int_{B(x_0,t)} u_m^{p-q}u^q \, d\mu\] 
and combining estimates \eqref{eqstartlayercake.eq}, \eqref{boundfortwocor.eq} and \eqref{boundforonecor.eq}, we may summarize our estimates as 
\begin{equation*}
\varphi(r_0) \le C \mu(NB) \alpha_p(NB) \widetilde{C}_N^{(p-q)(\ell + j)} + \epsilon_p \varphi(\rho_0) + C \epsilon_p^{-1} \int_{NB} f^p\, d\mu + C \epsilon_p^{-1} \int_{NB} g^p \, d\mu,
\end{equation*}
whenever $R\le r_0 < \rho_0 \le NR$ and $N^{\ell}(\rho_0 - r_0) = R$ and where $j$ depends at most on $N$, see Step~2. For notational convenience  set 
\begin{align*}
M_1 &:=  C \mu(NB) \alpha_p(NB),\\
M_2 &:=  C \int_{NB} f^p\, d\mu + C \int_{NB} g^p \, d\mu,
\end{align*}
so that {
\begin{align}
\label{rdy2iterate.eq} 
\varphi(r_0) \leq M_1 \widetilde{C}_N^{(p-q)\ell} + \epsilon_p \varphi(\rho_0) + \epsilon_p^{-1} M_2.
\end{align}
Now, we set up an iteration scheme to conclude: We fix $K \in \N$ large enough to guarantee $\sum_{\ell = 0}^\infty N^{-K \ell} \leq N$, initiate with $t_0 := R$ and put $t_{\ell+1} := t_{\ell} + N^{-K (\ell+1)}R$ for $\ell = 0,1,\dots$. Then $R\le t_\ell < t_{\ell + 1} \le NR$ and $N^{K(\ell+1)}(t_{\ell+1} - t_\ell) = R$ so that
\begin{align}
\varphi(t_\ell) \le M_1  \widetilde{C}_N^{(p-q)K(\ell+1)} + \epsilon_p^{-1} M_2 +  \epsilon_p \varphi(t_{\ell + 1}).
\end{align}
Iterating the above inequality we obtain for any $\ell_0 \in \N$
\begin{equation*}
\begin{split}
\varphi(t_0) &\le M_1\sum_{\ell = 1}^{\ell_0} \widetilde{C}_{N}^{(p-q)K\ell} \cdot \epsilon_p^{\ell -1} + M_2\sum_{\ell = 1}^{\ell_0} \epsilon_p ^{\ell - 2 } + \epsilon_p^{\ell_0}\varphi(t_{\ell_0})
\\& \le   C   M_1 + C M_2 + C\epsilon_p^{\ell_0}\varphi(NR)
\end{split}
\end{equation*}
provided that $\epsilon_p \leq (2\widetilde{C}_N^{(p-q)K})^{-1} \leq 1/2$ by now fixing $p \in (q,2q)$ with $p-q$ small enough.}

Noting that $\varphi(NR) < \infty$ (by truncation of $u$) and $t_0 = R$, we may let $\ell_0 \to \infty$ above to conclude
$$\varphi(R) \le CM_1 + C M_2.$$
Upon replacing $M_1,M_2$, $\alpha_p(NB)$ and $\varphi(R)$ we obtain
\begin{equation}
\begin{split}
\int_B u_m^{p-q} u^q\, d\mu & \le C \mu(NB) \big(a_u(NB) + (a_{f^q}(NB))^{1/q} + (NR)^\beta(a_{h^s}(NB))^{1/s}\big)^p 
\\& \qquad + C \int_{NB} f^p\, d\mu +  C \int_{NB} g^p \, d\mu.
\end{split}
\end{equation}
Dividing both sides of the inequality by $\mu(B)$, taking $p-$th roots and letting $m\to \infty$  we obtain  \eqref{GGeq2cor.eq} except for the presence of the term  $(\fint_{NB}  g^p\, d \mu)^{1/p}$. We handle this term using the definition of $g$ in terms of $h$ and the definition of $A_R$ in \eqref{eq:AR}, obtaining
\begin{align*}
\label{}
 \left(\fint_{NB}  g^p\, d \mu\right)^{1/p}   & = A_{R}\left(\fint_{NB}h^{ps/q} \, d\mu\right)^{1/p}   \\
    &  =  C_1 (NR)^\beta \left(\fint_{NB}h^s \, d\mu\right)^{1/s - 1/q} \left(\fint_{NB}h^{ps/q} \, d\mu\right)^{1/p} 
    \\
    & \le C_1  (NR)^\beta  \left(\fint_{NB}h^{ps/q} \, d\mu\right)^{q/ps- 1/p} \left(\fint_{NB}h^{ps/q} \, d\mu\right)^{1/p}
    \\
    & = C_1 (NR)^\beta \left(\fint_{NB}h^{ps/q} \, d\mu\right)^{q/sp}. \qedhere
\end{align*}
 \end{proof}
 
\section{Quasi-metric spaces}
\label{sec:quasi}

In Theorem \ref{GeneralGehering.cor}, it is possible to relax the structural assumption of $(X,d,\mu)$ being a metric space by allowing the constant $K$ in the quasi triangle inequality to take values greater than one. The proof of Theorem \ref{GeneralGehering.cor} does not carry over as such (or rather it becomes very technical) but we can take advantage of the fact that every quasi-metric ($K > 1$) is equivalent to a power of a proper metric ($K = 1$). See \cite{AIN1998,MS1979,PS2009}. The following proposition is from \cite{PS2009}. 
 
\begin{proposition}
\label{prop:quasimetric}
Let $(X,\rho)$ be a quasi-metric space and let $0 < \delta \le 1$ be given by
$(2K)^ \delta = 2$. Then there is another quasi-metric $\tilde{\rho}$ such that $\tilde\rho^{\delta}$ is a metric and for all $x,y \in X$, 
\[ E^{-1} \rho(x,y) \leq  \tilde{\rho} (x,y) \leq E \rho(x,y), \]
where $E \geq 1$ is a constant only depending on the quasi triangle inequality constant of $\rho$.
\end{proposition}

{With Proposition~\ref{prop:quasimetric} at hand, the following theorem is a straightforward consequence of its metric counterpart. For the reader's convenience and  since akin reductions to the metric case will be used at other occasions in this paper, we present the full details here.}

\begin{theorem}\label{GeneralGehering.quasi}
Let $(X,\rho,\mu)$ be a space of homogeneous type. Let {$s, \beta >0$ and $q>1$ be such that $s<q$} and $\beta \ge  D(1/s-1/q)$ where $D$ is any number satisfying \eqref{eq:homdim}. Let $N>1$ and  $(\alpha_k)_{k \geq 0}$ be a non-increasing sequence of positive numbers with $\alpha := \sum_k \alpha_k  < \infty$. Define 
\[a_u (B) := \sum_{k=0}^{\infty} \alpha_k \fint_{N^{k} B} u \, d \mu  \] 
for $u \geq 0$ locally integrable and $B$ a quasi-metric ball. 

Suppose that $u, f, h \ge 0$ with $u^q, f^q, h^s \in L^1_{loc}(X, d\mu)$ and there exists a constant $A$ such that for every ball $B = B(x, R)$, 
\begin{equation}
\label{eq:quasi1} 
\left(\fint_B u^q \, d\mu\right)^{1/q} \le Aa_{u}(B) + (a_{f^q}(B))^{1/q} + R^\beta(a_{h^s}(B))^{1/s}.
\end{equation} 
Then there exists $p > q$ depending on $\alpha_0, \alpha, A, q, s,K,N$ and $C_d$ such that for all balls $B$, 
\begin{equation}
\label{eq:quasi2} 
\begin{split}
\left(\fint_{ B } u^p \, d\mu\right)^{1/p} &\lesssim \tilde{a}_{u}({N}B) + (\tilde{a}_{f^q}({N}B))^{1/q} + R^\beta(\tilde{a}_{h^s}({N}B))^{1/s} 
\\& \qquad+ \left(\fint_{N E^{2/\delta} B}f^p \, d\mu\right)^{1/p} + R^\beta\left(\fint_{ N E^{2/\delta} B}h^{ps/q} \, d\mu\right)^{q/sp}.
\end{split}
\end{equation} 
Here, $\tilde{a}_u$ is obtained from $a_u$ by replacing the sequence $\alpha_k$ with $\alpha_{\max(0, k - j_0 - j_1)}$, where $j_0$ and $j_1$ are the minimal integers with $E^{2} \leq N^{j_0}$ and $E^{2/\delta} \leq N^{j_1}$ and $E, \delta$ are the constants from Proposition \ref{prop:quasimetric}.
The implicit constant depends on $\alpha_0,\alpha,  A, q, s, \beta, K, N$ and  $C_d$. \end{theorem}

\begin{remark}
 The same remarks as after Theorem \ref{GeneralGehering.cor} apply. 
\end{remark}

\begin{proof}
Let $d$ be the metric so that $d^{1/ \delta}$ with $\delta \in (0,1]$ is equivalent to $\rho$, provided by Proposition \ref{prop:quasimetric}. Then there is a constant $E > 1$ only depending on the quasi-metric constant $K$ of $\rho$ so that 
\[B^{\rho}(x,r) = \{z: \rho(z,x) < r \} \subseteq \{z: E^{-1} d^{1/\delta}(z,x) < r \} = B^{d}(x,(Er)^{\delta}) \]
and 
\[B^{\rho}(x,r) = \{z: \rho(z,x) < r \} \supseteq \{z: E d^{1/\delta}(z,x) < r \} = B^{d}(x,(E^{-1}r)^{\delta}).\]
In total $B^{d}(x,(E^{-1}r)^{\delta}) \subseteq B^{\rho}(x,r) \subseteq B^{d}(x,(Er)^{\delta})$. Consequently $\mu$ is doubling with respect to the metric $d$, and we also see that the hypothesis \eqref{eq:quasi1} implies that
\begin{equation*}
\begin{split}
\left(\fint_{ B^{d} (x, (E^{-1}R )^{ \delta}  ) } u^q \, d\mu\right)^{1/q} & \lesssim \sum_{k=0}^{\infty} \alpha_k \fint_{ N^{\delta k}  E^{\delta}  B^{d} (x,  R ^{ \delta}  )} u \, d \mu  \\
&\qquad + \left ( \sum_{k=0}^{\infty} \alpha_k \fint_{N^{\delta k}  E^{\delta}  B^{d} (x,  R ^{ \delta}  )} f^{q} \, d \mu \right )^{1/q} \\
& \qquad \qquad + R^{   \beta    } \left ( \sum_{k=0}^{\infty} \alpha_k \fint_{N^{\delta k}  E^{\delta}  B^{d} (x,  R ^{ \delta}  )} h^{s} \, d \mu \right)^{1/s}
\end{split}
\end{equation*}
holds for all $x \in X$ and $R > 0$. Setting $R' := (E^{-1}R)^{\delta}$, we can rewrite this as 
\begin{equation*}
\begin{split}
\left(\fint_{ B^{d} (x, R') } u^q \, d\mu\right)^{1/q} & \lesssim \sum_{k=0}^{\infty} \alpha_k \fint_{ N^{\delta k}    B^{d} (x, E^{2\delta} R ' )} u \, d \mu  \\
&\qquad + \left ( \sum_{k=0}^{\infty} \alpha_k \fint_{N^{\delta k}   B^{d} (x, E^{2\delta} R' )} f^{q} \, d \mu \right )^{1/q} \\
& \qquad \qquad + (R')^{   \beta / \delta   } \left ( \sum_{k=0}^{\infty} \alpha_k \fint_{N^{\delta k}    B^{d} (x, E^{2\delta} R' )} h^{s} \, d \mu \right)^{1/s}.
\end{split}
\end{equation*}

We set $N' := N^{\delta}$. Then $j_0$ is the smallest positive integer so that $E^{2\delta} \leq (N')^{j_0}$. We note that 
\begin{align*}
\sum_{k=0}^{\infty} \alpha_k \fint_{ (N')^{ k}    B^{d} (x, E^{2\delta} R ' )} u \, d \mu
	& \lesssim \sum_{k=0}^{\infty} \alpha_k \fint_{ (N')^{ k + j_0}    B^{d} (x, R ' )} u \, d \mu \leq \sum_{k=0}^{\infty} \alpha_k' \fint_{ (N')^{ k }    B^{d} (x, R ' )} u  d \mu,
\end{align*}
where $\alpha_k' := \alpha_{\max(k-j_0, 0)}$. Analogous estimates hold with $f^q$ and $h^s$ in place of $u$. Finally, we set $\beta' := \beta /\delta$ so that \eqref{GGeq1cor.eq} is satisfied in the metric space $(X,d,\mu)$ with $(\alpha_k')_{k}, \beta', N'$ replacing $(\alpha_k)_{k}, \beta ,N$ there. We also have control over the homogeneous dimension of $(X,d,\mu)$. Indeed, for $x \in X$ and $R > r$, we see that {$ER > E^{-1}r$ and therefore}
\[\frac{\mu( B^{d}(x, R^{\delta}))}{\mu( B^{d}(x, r^{\delta} ))} \leq \frac{\mu( B^{\rho}(x,ER))}{\mu( B^{\rho}(x, E^{-1}r) )} \lesssim \left(\frac{E^{2}R}{r} \right)^{ D} , \]
where $D$ is a number satisfying \eqref{eq:homdim} for $(X,\rho,\mu)$. 

It follows that $D'=D\delta^{-1}$ satisfies \eqref{eq:homdim} for $(X,d,\mu)$. As a consequence, $\beta' = \beta / \delta \geq D'(1/s-1/q)$, and we can apply Theorem \ref{GeneralGehering.cor}.

We obtain
\[\begin{split}
\left(\fint_{  B^{d} } u^p \, d\mu\right)^{1/p} &\lesssim \sum_{k=0}^{\infty} \alpha_k' \fint_{ (N')^{k+1} B^{d} } u \, d \mu + \left ( \sum_{k=0}^{\infty} \alpha_k' \fint_{ (N')^{k+1} B^{d}} f^{q} \, d \mu \right )^{1/q} \\
& \quad + (R')^{\beta'    } \left ( \sum_{k=0}^{\infty} \alpha_k' \fint_{ (N')^{k+1} B^{d}} h^{s} \, d \mu \right)^{1/s}
\\& \qquad+ \left(\fint_{ N'  B^{d}}f^p \, d\mu\right)^{1/p} + (R')^{\beta'}\left(\fint_{ N' B^{d}}h^{ps/q} \, d\mu\right)^{q/sp}
\end{split}\]
for all balls $B^{d}$ with radius $R'$. Note that $R'$ is arbitrary. Comparing the $d$-balls with $\rho$-balls once again, we see that $B^\rho(x, (E^{-1} r)^{1/\delta}) \subset B^{d}(x,r) \subset B^\rho(x, (E r)^{1/\delta})$. Arguing as in the beginning of the proof and denoting $R = (E^{-1} R')^{1/\delta}$, we can get back to an inequality in the quasi-metric space $(X,\rho,\mu)$: {We only need to recall that $N' = N^{\delta}$, that $j_1$ is the smallest integer so that $E^{2/\delta} \leq N^{j_1}$ and set $\alpha_k'' := \alpha_{\max(0, k- j_1)}' = \alpha_{\max(0, k - j_0 - j_1)}$}. This together with the doubling condition implies that \eqref{eq:quasi2} holds.
\end{proof}

\section{Variants}
\label{sec:variants}

One might wonder whether one can use in the proof of Theorem~\ref{GeneralGehering.cor} the fractional maximal operator $M^{\beta s}$ where \[M^{\beta}v(x):= \sup_{B\ni x}    r(B)^{ \beta } \fint_{B} |v|, \quad \ x \in X,\,  \beta>0,  \] 
{to control the terms stemming from $h^s$ more efficiently. (Here $r(B)$ denotes the radius of $B$}.) However, this operator has no boundedness property in this generality and one has to assume \emph{volume lower bound} in the following sense: 
\begin{equation}
\label{eq:lowerbound}
\exists Q >0 \ : \ \forall\ \mathrm{balls} \ B,  \quad \mu(B)\gtrsim r(B)^Q.
\end{equation} 

\begin{lemma}\label{lem:frac} Let $(X,\rho,\mu)$ be a space of homogeneous type. Assume that
the volume has a lower bound with exponent $Q > 0$. Then $M^{\beta}$ is bounded  from $L^p(X)$ to $L^{p^*}(X)$ when $ 1<p$ and $0 < \beta < Q/p $ with $p^*= \frac{pQ}{Q-\beta p}$. For $p=1$, it is weak type $(1, 1^*)$. 
\end{lemma}

\begin{proof} See e.g. Section~2 in \cite{HKNT} {for a simple proof on metric spaces with doubling measure that applies \emph{verbatim} in the quasi-metric setting. In fact,} the result follows from the inequality {$M^{\beta }v(x) \lesssim Mv(x)^{1-\beta/Q} \|v\|_{1}^{\beta/Q}$} using the lower bound, the uncentered maximal function $M$ and interpolation. 
\end{proof}
 
{We obtain the following variant in the presence of a volume lower bound.} 
 
\begin{theorem}\label{homers.4.3.thrm} Let $(X,\rho,\mu)$ be a space of homogeneous type having a volume lower bound with exponent $Q$. Let {$s>0$, $\beta \geq 0$ and $q>1$ be such that $s<q$ and $\beta \le Q (1/s -1/q)$}. 
Suppose that $u, f, h \ge 0$ with $u^q, f^q, h^s \in L^1_{loc}(X,d\mu)$ and \eqref{eq:quasi1} holds. Then there exists $p > q$ such that \eqref{eq:quasi2} holds with  $R^\beta(\fint_{{NE^{2/\delta} B}} h^{ps/q} \, d\mu)^{q/sp}$ replaced by $\mu(B)^{\beta/Q}(\fint_{{NE^{2/\delta}B}} h^{p_*} \, d\mu)^{1/p_*}$ where $p_*= \frac{pQ}{Q+\beta p}$.
\end{theorem}

\begin{proof}  
{It suffices to give a proof for $(X,\rho,\mu)$ a metric space with doubling measure. Then we can apply the general reduction argument from the previous section. In this regard, we note that if $\rho$ is equivalent to $d^{1/\delta}$, then $d$ has lower volume bound with exponent $Q':= Q/\delta$.} 

For any $p>q$ set $\sigma := \frac{pQ}{s(Q + \beta p)} = \frac{p_*}{s}$. Note the condition $\beta q s \le Q(q-s)$ ensures for all $p>q$ the bound $\beta p s < Q(p - s)$. {Hence $\sigma > 1$}. Now we indicate the changes in the proof of Theorem~\ref{GeneralGehering.cor}. 

One does not introduce the function $g$ in Step~1 and the function $G$ in Step~4 becomes $H(x,t)=r^\beta \big(\fint_{B(x,t)} h^s \, d \mu\big)^{1/s}$. The choice of $\lambda_{0}$ is similar and {then we can follow the argument until we need to estimate $II_{3}$ in Step~6. Here we now have} 
\begin{align*}
  {II_3 }
&= C(p-q)\int_0^m \lambda^{p-1} \mu(\{ M^{\beta s}(h^s\mathbbm{1}_{B_{\rho_0}}) > (\lambda/3)^s\}) \, d\lambda \\
& \lesssim  \int_X (M^{\beta s}(h^s\mathbbm{1}_{B_{\rho_0}}))^{p/s} \, d\mu \\ 
& \lesssim \left(\int_{NB } h^{s\sigma}\, d\mu\right)^{p/(s\sigma)}
{= \left(\int_{NB } h^{p_*}\, d\mu\right)^{p/p_*}}
\end{align*}
{by definition of $\sigma$} and we used the $L^{\sigma}(X) \to L^{p/s}(X)$ boundedness of $M^{\beta s }$ from Lemma \ref{lem:frac}. Recall near then end of Theorem \ref{GeneralGehering.cor} we divide by $\mu(B)$ then take $p$-th roots. Thus, the power of $\mu(B)$ in front of $(\fint_{NB} h^{p_*} \, d\mu)^{1/p_*}$  comes from the equality
$$ \mu(B)^{-1} \left(\int_B h^{p_*}\, d \mu \right)^{\frac{p}{p_*}} = \mu(B)^{\beta p/Q} \left(\fint_B h^{p_*} \, d\mu\right)^{\frac{p}{p_*}}.$$
\end{proof}
  
\begin{remark}
\label{rem:homers.4.3.thrm}
Assume $\beta=1$, $s=\frac{2n}{n+2}$ and  $q=2$ in the Euclidean space $\RR^n$ with Lebesgue measure, which is typical of elliptic equations. {Then the Lebesgue exponent for $h$ in Theorem~\ref{GeneralGehering.quasi} is $\frac{ps}{q}= \frac{pn}{n+2}$ while above we get $p_*= \frac{pn}{n+p}$, which is smaller. If $\beta =0$, then $p_* = p$. Of course the interest is to have $\beta$ as large as possible so that $p_*$ is as small as possible, but in applications to PDEs the value of $\beta$ is usually not free to choose but determined by scaling arguments. Finally note that the admissible ranges for $\beta$ in the two theorems are almost complementary in this example: Indeed, since $D=Q=n$, we have $\beta \geq n(1/s - 1/q)$ in Theorem~\ref{GeneralGehering.cor} and $\beta \leq n(1/s-1/q)$ in Theorem~\ref{GeneralGehering.quasi}.}
\end{remark}

Another variant is to replace powers of the radius by powers of the volume {already in the assumption} and then no further hypothesis on the measure is required.

\begin{theorem}Let $(X,\rho,\mu)$ be a space of homogeneous type. {Let $s>0$, $\gamma \geq 0$ and $q>1$ be such that $s<q$ and $\gamma \le 1/s -1/q$.}
Suppose that $u, f, h \ge 0$ with $u^q, f^q, h^s \in L^1_{loc}(X,d\mu)$ and \eqref{eq:quasi1} holds with  $R^\beta(a_{h^s}(B))^{1/s}$ replaced by $\mu(B)^\gamma(a_{h^s}(B))^{1/s}$. 
Then there exists $p > q$ such that \eqref{eq:quasi2} holds with $R^\beta(\fint_{{N E^{2/\delta} B}}h^{ps/q} \, d\mu)^{q/sp}$ replaced with $\mu(B)^\gamma (\fint_{{N E^{2/\delta B} }} h^{s\sigma}\, d\mu)^{1/s\sigma}$, where  $s\sigma=\frac{p}{1+\gamma p}$. 
\end{theorem}

\begin{remark} Note that one can take $\gamma=0$ in which case $s\sigma=p$. {In accordance with Remark~\ref{rem:homers.4.3.thrm}} we note that the higher $\gamma$ the smaller the integrability needed on $h$.
\end{remark}

\begin{proof} {Once again it suffices to treat the metric case}. The modification to the proof of Theorem~\ref{GeneralGehering.cor} are the same as in the above argument, except for now using instead of $M^{\beta s}$  the modified fractional maximal operator 
$\tilde M^{\gamma s}$, $0\le \gamma<1$,  where
\[\tilde M^{\gamma s}v(x):= \sup_{B\ni x} \  (\mu(B))^{\gamma s} \fint_{B} |v|, \quad \ x\in X.   \]
It maps $L^\sigma(X)$ into $L^{\frac{\sigma}{1-\gamma s \sigma}}(X)$ when $1<\sigma$ and $ \gamma s \sigma < 1 $,  see Remark~2.4 in \cite{HKNT}. \end{proof}

\section{Global integrability}
\label{sec:global}

A typical application of Gehring's lemma is to prove higher integrability locally and globally. To extract a conclusion at the level of global spaces $L^{p}(X)$, we need some further hypotheses. We say that  the space of homogeneous type $(X,\rho,\mu)$ is \emph{$\phi$-regular} if it satisfies
\[ \phi(r)  \sim \mu( B(x,r) )\] 
for all $x \in X$ and $r  >0$, where $\phi: (0,\infty) \to (0,\infty )$ is a non-decreasing function with $\phi(r) > 0$ and $\phi(2r) \sim \phi(r)$ for $r > 0$. An important subclass of such spaces are the \emph{Ahlfors--David regular} metric spaces where $\phi(r) = r^{Q}$ for some $Q > 0$. The case of local and global different dimensions which occur on connected nilpotent Lie groups (see \cite{VSC}) is also covered with $\phi(r) \sim r^{d}$ for $r\le 1$ and $\phi(r) \sim r^{D}$ for $r\ge 1$.

\begin{theorem}
\label{thm:global1}
Let $(X,\rho,\mu)$ be a $\phi$-regular space of homogeneous type. In addition to assumptions of Theorem \ref{GeneralGehering.quasi}, suppose that $u^q, f^q, h^s \in L^1(X, d\mu)$. Then 
\[ \norm{u}_{L^{p}(X)} \lesssim  \norm{u}_{L^{q}(X)} + \norm{f}_{L^{p}(X)}+ \norm{h}_{L^{ps/q}(X)}   \]
with the implicit constant depending on  {$u, f, h$ only through the parameters quantified in the assumption.}
\end{theorem} 

\begin{proof} {For the sake of simplicity let us assume $N=2$ in the statement of Theorem~\ref{GeneralGehering.quasi}. We shall see in Section~\ref{sec:dilation} below that upon changing the sequence $(\alpha_k)_k$ we can do so without loss of generality. Alternatively, we could also adapt the following argument to cover the general case.}

Take any $R > 0$ and choose a maximal $R$ separated set of points $\{x_i\}$, that is, $\rho(x_i, x_j) \ge R$ for all $i \neq j$ and for every $y \in X$ there exists $x_i$ such that $\rho(y,x_i) < R$. Since we assume that $X$ is doubling, such a collection necessarily has only finitely many members in any fixed ball, hence, it is countable. The balls $B_i := B(x_i,R)$ cover $X$, and there is $C$ only depending on $K$ and $C_d$ such that 
\[\sum_i \mathbbm{1}_{B_i}(x) \leq C \]
for every $x \in X$. Also the balls $(2K)^{-1}B_{i}$ are disjoint. Further, we have 
\begin{equation}
\label{eq:2kcovering}
\sum_i \mathbbm{1}_{2^kB_i}(x) \sim  \frac{\phi(2^kR)}{\phi(R)}
\end{equation} 
for every $x \in X$ and every integer $k\ge 1$. Indeed, fix $k\ge 1$ and $x\in X$. We can assume that {$2^{k-1}\ge K$,} since otherwise {we can just use that the left- and right-hand sides are comparable to constants depending only on $K$, $C_d$ and $\phi$}. Let $I_{x}$ be the set of $i$ giving a non zero contribution, and $N_{x}$ be the cardinal of $I_{x}$, that is, the value of the sum. Clearly, $N_{x}$  is not exceeding the number of $i$ for which $\rho(x,x_{i})\le 2^k R$. As the balls $(2K)^{-1}B_{i}$, $i\in I_{x}$, are disjoint and contained in $B(x, K(2^k+(2K)^{-1})R)$, we have 
$$
\phi(R/2K) N_{x} \lesssim \sum_{i\in I_{x}} \mu((2K)^{-1}B_{i}) \lesssim \mu(B(x, K(2^k+(2K)^{-1})R)) \lesssim \phi(K 2^{k+1}R).
$$
Also the balls $B_{i}$, $i\in I_{x}$, cover $B(x,(K^{-1}2^k-1) R)$, hence
\[
\begin{split}
\phi(R/2K) N_{x} \gtrsim &\sum_{i\in I_{x}} \mu((2K)^{-1}B_{i}) \gtrsim \sum_{i\in I_{x}} \mu(B_{i}) \gtrsim  \mu(B(x, (K^{-1}2^k-1) R)) \\
& \gtrsim \phi((K^{-1}2^k-1) R) {\geq \phi(K^{-1}2^{k-1}R)}
\end{split}
\]
{by the assumption on $k$}. The claim follows using the comparability $\phi(K2^{k+1}R)\sim {\phi(K^{-1} 2^{k-1}R)} \sim   \phi(2^k R)$ and $\phi(R/2K)\sim \phi(R)$. 

Applying Theorem \ref{GeneralGehering.quasi} in each $B_i$ and {denoting by $\tilde \alpha_k := \alpha_{\max(0,k-j_0-j_1)}$ the summable sequence appearing in the conclusion of that theorem},
 we see that
\begin{align*}
\phi(R)^{-1} \int_{X}  & u^{p} \, d \mu  \lesssim  \sum_{i} \fint_{B_i} u^{p} \, d \mu   \lesssim    \sum_i \Big(\tilde a_{u}(2B_{i})^p + \tilde a_{f^p}(2B_{i})+ R^\beta \tilde a_{h^{ps/q}}(2B_{i})^{q/s} \Big) \\		   
			& \qquad =: I + II + III,
\end{align*} 
where we used $\phi$-regularity, H\"older's inequality $\tilde a_{f^q}(2B_{i})^{1/q} \lesssim \tilde a_{f^p}(2B_{i})^{1/p}$ and absorbed the term with $f^p$ in \eqref{GGeq2cor.eq} in $\tilde a_{f^p}(2B_{i})$ and similarly for the terms with $h$. 
Let us treat $III$: {From the continuous embedding $\ell^1(\N) \subset \ell^{q/s}(\N)$ and \eqref{eq:2kcovering} we obtain}
\begin{align*}
\label{}
\sum_{i } \tilde a_{h^{ps/q}}(2B_{i})^{q/s} &  \le \sum_{i }\bigg(\sum_{k=0}^{\infty} \tilde \alpha_k \fint_{2^{k+1} B_{i}} h^{ps/q} \, d \mu\bigg)^{q/s}  
       \\
    &   \le \bigg(\sum_{i } \sum_{k=0}^{\infty} \tilde \alpha_k \fint_{2^{k+1} B_{i}} h^{ps/q} \, d \mu\bigg)^{q/s}  \\
    &  \le \bigg( \sum_{k=0}^{\infty} \tilde \alpha_k \sum_{i }  \fint_{2^{k+1} B_{i}} h^{ps/q} \, d \mu\bigg)^{q/s}
\\
& \lesssim  \bigg( \sum_{k=0}^{\infty} \tilde \alpha_k\   \phi(2^{k+1} R)^{-1} \int_X \sum_{i } \mathbbm{1}_{2^{k+1} B_{i}} h^{ps/q} \, d \mu\bigg)^{q/s} \\
& \lesssim \bigg(\tilde \alpha \phi(R)^{-1} \int_{X}h^{ps/q} \, d \mu\bigg)^{q/s},
\end{align*}
where $\tilde \alpha := \sum_k \tilde \alpha_k$. Doing the same for $I$ {and} $II$, implies 
 \begin{equation}
\label{eq:R}
\norm{u}_{L^{p}(X)} \lesssim  \phi(R)^{1/p-1/q}\norm{u}_{L^{q}(X)} + \norm{f}_{L^{p}(X)}+ R^\beta \phi(R)^{(1/p)-(q/sp)} \norm{h}_{L^{ps/q}(X)}. 
\end{equation} 
Note that since $R$ is fixed ($R=1$ for example), this concludes the proof.  
\end{proof}

\begin{remark} We note that the implicit constant in \eqref{eq:R} does not depend on $R$. 
 If $h=0$, then we may let $R\to \infty$ as $p>q$ and obtain $\norm{u}_{L^{p}(X)} \lesssim   \norm{f}_{L^{p}(X)}$. 
\end{remark}

We have  {a global analogue of Theorem~\ref{homers.4.3.thrm}, which might be of independent interest as it can be proved directly in the quasi-metric setting without recursing to Section~\ref{sec:quasi}.}

\begin{theorem}
\label{thm:global_volumelowerbound}
Let $(X,d,\mu)$ be a space of homogeneous type having volume  lower bound with exponent $Q$. {Let $s>0$, $\beta \geq 0$ and $q>1$ be such that $s<q$ and $\beta \le Q (1/s -1/q)$. Suppose that $u^q, f^q, h^s \in L^1(X, d\mu)$ and that \eqref{eq:quasi1} holds}. Then there exists  $p>q$ such that 
\[ \norm{u}_{L^{p}(X)} \lesssim    \norm{f}_{L^{p}(X)}+ \norm{h}_{L^{p_{*}}(X)},   \]
{where $p_* = \frac{pQ}{Q+\beta p}$. The implicit constant depends on $u, f, h$ only through the parameters quantified in the assumption.}
\end{theorem}

{As for the choice of $\beta$, the same comments as in Remark~\ref{rem:homers.4.3.thrm} apply.}

\begin{proof} We indicate the modification to the argument of Theorem \ref{GeneralGehering.cor}, which, as said, works directly in the quasi-metric setting for this result. This is basically the one in \cite{ABES2017}. 

There is no need for the first and second steps and the proof begins as in Step~3 without the balls $B_{r_{0}}$ and $B_{\rho_{0}}$ and we have 
$$
\int_{X} u_m^{p-q}u^q \, d\mu  = (p-q)\int_{0}^m \lambda^{p -q - 1} u^q( \{u > \lambda\})\, d\lambda.
$$
There is also no need for a threshold $\lambda_{0}$ and we set  for $x\in X$ and $r>0$, 
\begin{align*}
U(x,r)	: =  a_{u}(B(x,r)), \quad
F(x,r)	: = (a_{f^q}(B(x,r))^{1/q}, \quad
H(x,r)	: = r^\beta (a_{h^s}(B(x,r)))^{1/s},
\end{align*}
and for $\lambda > 0$, we denote $U_\lambda :=  \{u > \lambda\}$. {Note that without loss of generality we may assume $\alpha_0 \geq 1$ right from the start as this only increases the right-hand side of our hypothesis \eqref{eq:quasi1}}. Thus,   
\[\liminf_{r \to 0} \big( U(x,r) + F(x,r) + H(x,r) \big) \geq u(x) \] 
for almost every $x$ because already {the first term in $a_{u}(B(x,r))$ tends to $\alpha_0 u(x) \geq u(x)$}. We define $\tilde U_{\lambda}$ as the subset of $U_{\lambda}$ where this holds. Note that 
\[\lim_{r \to \infty} \big( U(x,r) + F(x,r) + H(x,r) \big) = 0 \]
for all $x$  using the global assumptions on $u,f,h$. For the term with $h$, this follows from $H(x,r)^s \lesssim r^{\beta s -Q}  \int_{X} h^s \, d\mu$, provided $\beta s < Q$, which holds under our assumption.
For $x\in \tilde U_{\lambda}$, we can define the stopping time radius
\begin{equation*}
r_{x}:= \sup\{ r > 0:  U(x, r) + F(x,r) + H(x,r) > \lambda \}.
\end{equation*}
Remark that  $\sup_{x\in \tilde U_{\lambda}} r_{x} < \infty$. Indeed, at $r=r_{x}$, {$U(x,r) + F(x,r) + H(x,r) = \lambda$ and therefore} either $U(x, r)\ge \lambda/3$  or  $F(x,r) \ge \lambda/3$ or $H(x,r) \ge \lambda/3$. In the last case, we obtain $r^{Q-\beta s}(\lambda/3)^s \lesssim \int_{X} h^s \, d\mu <\infty$. The other cases also give us a bound on $r$.  By the Vitali covering lemma, there exists a countable collection of balls $\{B(x_i,r_{x_{i}})\} = \{B_i\}$ such that $\tfrac{1}{V}B_i$ are pairwise disjoint and $\tilde U_\lambda \subset \cup_i B_i$. (Usually, $V=5$ but our metric is only a quasi-metric in which case the Vitali covering lemma still holds but with a larger constant $V$ depending on $K$, and we apply it to the covering  $\tilde U_\lambda \subset \cup_{x \in \tilde U_\lambda} B(x, V^{-1} r_{x})$. {A direct way to see this is by the technique in Section~\ref{sec:quasi}}.)
Now, using the hypothesis for each $B_{i}$ and pairwise disjointness of the balls $\tfrac{1}{V}B_i$,
\begin{align*}
u^q(\tilde U_{\lambda})  &\le \sum_i  u^q({B_i}) 
		\le \sum_i \mu(B_i)\big(Aa_u(B_i) + (a_{f^q}(B_i))^{1/q} + r_i^\beta (a_{h^s}(B_i))^{1/s}\big)^q \\
		&= \sum_i  \mu(B_i)\lambda^{q} \lesssim V^D \sum_i  \mu(\tfrac 1 V B_i) \lambda^{q} \lesssim V^D  \mu \left( \cup_{i}  B_i \right) \lambda^{q},
\end{align*}
where $D$ is the homogeneous dimension.  The stopping time  implies
\begin{equation}
\label{uqbnd2corglobal.eq}
\cup_i  B_i \subset  \{ Mu \ge \lambda/3\}\cup  \{  M(f^q) \ge (\lambda/3)^q\}
 \cup \{ M^{\beta s}(h^s) \ge (\lambda/3)^s\}.
\end{equation}
From there  the estimates are as in Step~6 and we {use Lemma~\ref{lem:frac} for boundedness of $M^{\beta s}$}. We obtain
\begin{equation}\label{boundfortwocorglob.eq}
\int_{X} u_m^{p-q}u^q \, d\mu \le C(p-q) \int_{X}u_{m}^{p-1} u \, d \mu + C_p \int_{X} f^p\, d\mu +C_p \left(\int_{X} h^{p_{*}} \, d\mu\right)^{p/p_{*}},
\end{equation}
with $p_{*}$ as in the statement.  As $\int_{X}u_{m}^{p-1} u \, d \mu \le \int_{X}u_{m}^{p-q} u^q \, d \mu$ we can hide this term if  $p-q>0$ is small enough and then let $m\to \infty$.
\end{proof} 

\section{Self-improvement of the right-hand side}
\label{sec:sirhs}

We discuss here the change in the exponents in the tails on the right-hand side and subsequently the change of the dilation parameter $N$. Both induce change in the sequence $\alpha$. These remarks can be used to reduce some seemingly different properties to cases covered by Theorem \ref{GeneralGehering.quasi}.

\subsection{Exponent}
It is a direct consequence of the $\log$-convexity of the $L^{p}$ norms that if
\[\left(   \fint_{B} u^{p}\, d \mu \right)^{1/p} \lesssim \left( \fint_{B} u^{q} \, d \mu \right)^{1/q} \]
with $p > q$, then for every $s \in (0, q)$ we can write $1/q = \theta / p + (1- \theta) /s$ for some $\theta \in (0,1)$ and consequently
\[\norm{ u }_{L^{p}( B , \nu ) }  \lesssim \norm{ u }_{L^{q}( B, \nu  ) } \leq \norm{ u }_{L^{s }( B ,\nu  ) }^{(1- \theta)} \norm{ u }_{L^{p}( B, \nu ) }^{\theta}  \]
so that $\norm{ u }_{L^{p}( B , \nu) }  \lesssim \norm{ u }_{L^{s }( B ,\nu  ) }$. Here $\nu = d \mu / \mu(B)$. The same self-improving property holds true for the weak reverse H\"older inequality \cite{IN1985} and even for the reverse H\"older inequality with tails as we now show. To prove the claim for the inequality with tails, we use a modification of the argument from \cite{BCF2016}, Appendix B.

\begin{proposition}\label{prop:exponent}
Let $(X,\rho,\mu)$ be a space of homogeneous type. Let $  q  \in (0, p)$, $s_{0}, s_1,s_2 \in (0,q]$ with  $f^{s_1}, h^{s_2}\in L_{loc}^{1}(X)$ and set $\tau := {\min \big(\frac{s_{0}}{q}, \frac{s_1}{q}, \frac{s_2}{q} \big) }$. Let $(\alpha_k)_{k \geq 0}$ be a summable sequence {of strictly positive numbers. Let $N > 1$ and $\beta \geq 0$}. 
Let  $(\tilde\alpha_k)_{k \geq 0}, (\alpha_k^\sharp)_{k \geq 0}$ be summable sequences of non negative numbers {with $\tilde \alpha_0, \alpha_0^\sharp > 0$} and assume  
\begin{equation}
\label{eq:alpha_rhs}
 \sum_{k = 0}^{m} \tilde \alpha_k \alpha_{m-k}^\tau  \lesssim \tilde\alpha_m,  \qquad  \sum_{k = 0}^{m}  \alpha_k^\sharp \alpha_{m-k}  \lesssim \alpha_m^\sharp
\end{equation}
and 
\begin{equation}
\label{eq:alpha_rhs2}
 \sum_{k = 0}^{m} \tilde \alpha_k \alpha_{m-k}^{s_2/q} N^{(m-k) \beta s_2}  \lesssim \tilde\alpha_m. 
\end{equation}
Define $a_u(B), \tilde{a}_u(B), {a}^\sharp_u(B)$ as in \eqref{eq:aub} in terms of the  three respective sequences, for $u \geq 0$ locally integrable, $N > 1$ and $B$ a quasi-metric ball. 

Assume that  
\begin{equation}
\label{eq:assumption_rhs} 
\left(\fint_B u^p \, d\mu\right)^{1/p} \lesssim ( a_{u^{q}} (B))^{1 / q} + b(B),
\end{equation} 
where the implicit constant does not depend on $B$, with  
\begin{align*}
b(B)  = (a_{f^{s_1}}(B))^{1/s_1} + r(B)^\beta(a_{h^{s_2}}(B))^{1/s_2}.
\end{align*}
Then, for any ball $B$ for which ${a}^\sharp_{u^{q}}(B) < \infty$, one has
\begin{equation}
\label{eq:statement_rhs}
\left(\fint_B u^p \, d\mu\right)^{1/p} \lesssim ( \tilde { a}_{u^{s_{0}}} (B))^{1/s_{0}} + \tilde {b}(B),
\end{equation}
where $\tilde{b}$ is obtained by replacing $\alpha$ by $\tilde{\alpha}$ in the definition of $b$.
\end{proposition} 

\begin{remark}\label{withgeopropexp.rmk}
Note that, in contrast with the improvement of integrability, we do not need the non-increasing assumption on the sequence $\alpha$ for this proposition. The condition \eqref{eq:alpha_rhs} together with $\alpha_0^\sharp > 0$ implies $\alpha_{m}\lesssim \alpha_{m}^\sharp$ and similarly, since $\tau \leq 1$, we have $\alpha_m \lesssim \tilde \alpha_m$. {Hence we have to assume more than  $a_{u^{q}} (B)<\infty$. 
For example, with $\tau$ as above, if $\alpha_{k}= N^{-\gamma k}$ for $ \gamma>0$, then $\tilde \alpha_{k}= N^{-\gamma' \tau k}$ and $\alpha_{k}^\sharp=N^{-\gamma' k}$ work in the theorem for any $0<\gamma'<\gamma$ such that $\beta s_2 < \gamma s_2/q - \gamma' \tau$. In particular, decay $\gamma > \beta q$ is needed to obtain any improvement, which typically is hard to obtain in applications. On the other hand, if $\beta = 0$, we can improve the right-hand exponent by only paying an arbitrarily small amount of decay to replace $\gamma$ by $\gamma' < \gamma$.}  {The condition \eqref{eq:alpha_rhs2} takes into account the presence of $r(B)^\beta$ in \eqref{eq:assumption_rhs}.} 
Finally, the strict positivity of $\alpha_{k}$ rules out in particular the case where the $\alpha_{k}$ form a finite sequence, but in that case, the  argument in  \cite{BCF2016} already covers the situation. 
\end{remark}

\begin{proof}
Define 
\[ K (\delta, s_{0}) := \sup \frac{ ( a_{u^{q}} (B))^{1 / q} }{  (\tilde{ a }_{u^{s_{0}}}(B))^{1 / s_{0}} +  \tilde{b}(B) + \delta ( {a}^\sharp_{u^{q}}(B))^{1 / q} }, \]
where the supremum is taken on the set of balls $B$ such that  the denominator is finite.    Indeed, there is nothing to prove if the right-hand side of \eqref{eq:statement_rhs} is infinite, which is equivalent to the denominator being infinite since   we assume ${a}^\sharp_{u^{q}}(B) < \infty$.  As   $\alpha_{m} \lesssim \alpha_{m}^\sharp $, we have ${a}_{u^{q}}(B)\lesssim {a}^\sharp_{u^{q}}(B)$ and  the presence of $\delta>0$ guarantees that $K (\delta, s_{0})\lesssim \delta^{-1}$. We show a uniform bound in terms of $\delta$. {To this end we can of course assume $K(\delta,s_0) \geq 1$ since otherwise there is nothing to prove.}

Fix a ball $B$ with the above restriction.  Let $\theta \in (0,1)$ be such that 
\[\frac{1}{q} = \frac{\theta}{s_{0}} + \frac{1-\theta}{p}.\]
We see that
\[ \left( \fint_{B} u^{q} \, d \mu  \right)^{1/q} \leq \left( \fint_{B} u^{s_{0}}  \, d \mu  \right)^{\theta /s_{0}} \left( \fint_{B} u^{p}  \, d \mu  \right)^{(1-\theta)/p} .\]
Using \eqref{eq:assumption_rhs}, $b(B)\lesssim \tilde b(B)$ and $K(\delta,s_{0})\ge 1$, 
\begin{align*}
 &\left( \fint_{B} u^{q}  \, d \mu  \right)^{1/q} \lesssim \left( \fint_{B} u^{s_{0}} \, d \mu  \right)^{\theta /s_{0}} \left( ( a_{u^{q}} (B))^{1 / q} + {b}(B) \right)^{1- \theta} \\
 	& \lesssim \left( \fint_{B} u^{s_{0}}  \, d \mu \right)^{\theta /s_{0}} K(\delta, s_{0})^{1-\theta}  \left( (\tilde{ a }_{u^{s_{0}}}(B))^{1 / s_{0}} +  \tilde{b}(B) + \delta ( a_{u^{q}}^\sharp(B))^{1 / q} \right)^{1-\theta} \\
 	&\lesssim K(\delta, s_{0}) ^{1-\theta} \left( (\tilde{ a }_{u^{s_{0}}}(B))^{1 / s_{0}} +  \tilde{b}(B) + \delta ( a_{u^{q}}^\sharp(B))^{1 / q} \right) .
\end{align*}
We apply this last inequality to $N^kB$. This is possible provided $N^kB$ belongs to the same set of balls and this follows from the assumption on the sequences: For example,  using \eqref{eq:alpha_rhs}, 
$$a^\sharp_{u}(N^kB)= \sum_{j=k}^{\infty} \alpha_{j-k}^\sharp   \fint_{N^{j} B} u^{q}  \ d \mu
\lesssim \alpha_k^{-1}  \sum_{j=k}^{\infty} \alpha_{j}^\sharp   \fint_{N^{j} B} u^{q}  \ d \mu \le \alpha_k^{-1} a^\sharp_{u^q}(B)<\infty.$$ 
Similar calculations can be done for the other terms.  Thus, 
\begin{align*}
a_{u^{q}} (B)^{1/q}
=   \left(\sum_{k = 0}^{\infty} \alpha_k \fint_{N^{k }B } u^{q} \, d\mu \right)^{1/q}
&\lesssim K(\delta,s_{0}) ^{(1-\theta)}     \left( \sum_{k=0}^{\infty} \alpha_k  \tilde{ a }_{u^{s_{0}}}(N^{k} B)^{q/s_0} \right)^{1/q} \\
& \quad + K(\delta, s_{0}) ^{(1-\theta)} \left( \sum_{k=0}^{\infty} \alpha_k  \tilde b(N^{k}B )^q \right)^{1/q} \\
& \quad + K(\delta, s_{0}) ^{(1-\theta)} \delta \left( \sum_{k=0}^{\infty} \alpha_k {a}^\sharp_{u^{q}}(N^{k}B)  \right)^{1/q} \\\
&=: I + II + III.
\end{align*}
Each of the three sums is estimated similarly so we restrict our attention to the first one. Using {the continuous embedding $\ell^1(\N) \subset \ell^{q/s_0}(\N) $} and the properties of $\alpha$ in  \eqref{eq:alpha_rhs}, we compute 
\begin{align*}
I^{s_0} \leq \sum_{k=0}^{\infty} \alpha_k ^{s_{0} / q} \tilde{ a }_{u^{s_{0}}}(N^{k} B) 
& = \sum_{m=0}^{\infty}  \left(\sum_{ k = 0} ^{m} \alpha_k  ^{s_{0} / q} \tilde{\alpha}_{m-k} \right) \fint_{N^{m} B} u^{s_{0}}  \ d \mu \\
& \lesssim \sum_{m=0}^{\infty}   \tilde{\alpha} _{m}  \fint_{N^{m} B} u^{s_{0}}  \ d \mu 
= \tilde{ a }_{u^{s_{0}}}(  B).
\end{align*}
The same kind of argument applies to the remaining two terms so that
\begin{equation}
\label{eq:K}
a_{u^{q}} (B)^{1/q} \lesssim  K (\delta, s_{0})^{1-\theta}  \big( (\tilde{ a }_{u^{s_{0}}}(  B))^{1/s_{0}} + \tilde{ b }(  B) + \delta ( {a}^\sharp_{u^{q}} (B) )^{1/q} \big).
\end{equation}
{We remark that it is the part of $\tilde b$ involving $r(B)^\beta$ that requires us to use the strong condition \eqref{eq:alpha_rhs2}.}
 As   the right-hand side is finite, we readily obtain $K (\delta, s_{0}) \lesssim K (\delta, s_{0})^{1-\theta}$, therefore $K (\delta, s_{0}) \lesssim 1$. Now, all the bounds are independent of $\delta$, so we may send $\delta \to 0$ in \eqref{eq:K}. Plugging this inequality into \eqref{eq:assumption_rhs} concludes the proof of \eqref{eq:statement_rhs}.
\end{proof}

\subsection{Dilation}
\label{sec:dilation}

Another direction to which the reverse H\"older inequalities self-improve is the dilation parameter on the right-hand side of 
\[\left(   \fint_{B} u^{p}\, d \mu \right)^{1/p} \lesssim \left( \fint_{NB} u^{q} \, d \mu \right)^{1/q} . \]
Indeed, if such an inequality holds in a space of homogeneous type, then the similar inequality
\[\left(   \fint_{B} u^{p}\, d \mu \right)^{1/p} \lesssim \left( \fint_{CB} u^{q} \, d \mu \right)^{1/q} \]
holds for all balls with any $C  > K$ where $K$ is the quasi-metric constant. See for instance Theorem 3.15 in \cite{AHT2017}. The proof of this fact is based on a covering of $B$ by small balls whose $N$-dilates are still contained in $C B$ and applying the weak reverse H\"older inequality in each small ball individually. 

It is worth a remark that a change of geometry similar to the property just described can be carried out with the reverse H\"older inequality with tails. To formulate this technical remark, we introduce some notation. Given a sequence of positive numbers $\alpha = (\alpha_k)_{k\ge0}$ and  real numbers with $1<n \le m$,  we define the $(m,n)$-\emph{stretch} $S^{m,n} \alpha $  by 
\[(S^{m,n} \alpha)_j := \alpha_{k },  \quad   j \ge 0 :   \ m^{k-1} <   n^{j} \leq m^{k} .  \] 
We also define the $(m,n)$-\emph{regrouping} by 
\[ (R^{m,n} \alpha)_k := \sum_{j: m^{k-1} < n^{j} \leq m^{k} } \alpha_j\  + \beta_k , \quad k\ge 0, \]
where $\beta_k \ $ is a correction term. It makes each $(R^{m,n} \alpha)_k$ to be the sum of equally many terms and hence the regrouping of a non-increasing sequence remains non-increasing: The intervals
\[\left( (k-1) \frac{\ln m}{\ln n} , k \frac{\ln m}{\ln n}  \right] \]
contain $\ell$ or $\ell + 1$ integers when $\ell$ is the integer such that
\[ \ell <  \frac{\ln m}{\ln n} \leq \ell + 1.	 \] 
We set $\beta_k = 0$ if $\sum_{j: m^{k-1} <   n^{j} \leq m^{k} } 1 = \ell + 1$ and $\beta_k = \alpha_{\min\{{j}: n^{j} > m^{k-1} \} }$ otherwise.

For example, for $\gamma>0$,  the $(m,n)$-stretch of $(m^{-\gamma k})$ is (term-wise) comparable to $(n^{-\gamma k})$, and the $(m,n)$-{regrouping} of $(n^{-\gamma k})$ is (term-wise) comparable to  $(m^{-\gamma k})$. More generally, if $\alpha$ is summable, so are its stretch and regrouping. For the latter, it is obvious and for the former, is follows from bounding  the number of  possible repetitions by   $1+\frac{\ln m}{\ln n}$. In addition, if $\alpha$ is non-increasing so are its $(m,n)$-stretch $(m,n)$-{regrouping}.
  
\begin{proposition}\label{prop:dilation}
Let $(X,\rho,\mu)$ be a space of homogeneous type and let $(\alpha_k)_{k \geq 0}$ be a summable sequence of positive numbers. For $u \in L^{1}_{loc}(X)$, $u\ge 0$,  $N>1$, define
\[a_u(B) := \sum_{k=0}^{\infty} \alpha_k \fint_{N^{k} B} u \, d \mu. \]
Then for any $M > 1$, one has 
\[a_u(B) \lesssim \sum_{k=0}^{\infty} \beta_k \fint_{M^{k} B} u \, d \mu,\]
with, if $M > N$, $\beta = R^{M,N} \alpha $ and if  $M <  N$,  $\beta = S^{M^\ell,M} R^{M^{\ell},N}\alpha$ where $\ell$ is the least integer to satisfy $\ell \geq \ln N / \ln M$. 
\end{proposition} 

\begin{proof}
We start with $M > N$.  Then, 
\begin{align*}
 \sum_{k=0}^{\infty} \alpha_k \fint_{N^{k} B} u \, d \mu  
 	&= \sum_{k=0}^{\infty} \sum_{j: M^{k-1} <   N^{j} \leq M^{k} }  \alpha_j \fint_{N^{j} B} u \, d \mu  
 	\lesssim \sum_{k=0}^{\infty} \sum_{j: M^{k-1} <   N^{j} \leq M^{k} }  \alpha_j \fint_{M^{k} B} u \, d \mu  \\
 	&\le \sum_{k=0}^{\infty} (R^{M,N} \alpha)_k \fint_{M^{k} B} u \, d \mu 
\end{align*}
as claimed, using the doubling condition. 

Let now $M < N$. Assume first that there is an integer $\ell\ge 2$ such that $M^{\ell} = N$. Then we can write 
\begin{align*}
\sum_{k=0}^{\infty} \alpha_k \fint_{N^{k} B} u \, d \mu 
 	= \sum_{k=0}^{\infty} \alpha_k \fint_{M^{\ell k} B} u \, d \mu  
 	\leq \sum_{j=0}^{\infty} (S^{M^\ell, M}\alpha)_{j} \fint_{M^{j} B} u \, d \mu. 
\end{align*}
In general, we can find an integer $\ell \ge 2$ so that $M^{\ell-1} < N \le  M^{\ell} $ so that by the previous case ($M> N$) 
\[\sum_{k=0}^{\infty} \alpha_k \fint_{N^{k} B} u \, d \mu \lesssim \sum_{k=0}^{\infty}  (R^{M^{\ell},N}\alpha)_k \fint_{M^{\ell k} B} u \, d \mu \le \sum_{k=0}^{\infty}  (S^{M^\ell, M} R^{M^{\ell},N}\alpha)_j \fint_{M^{j} B} u \, d \mu .\] 
\end{proof}

\section{Extensions}
\label{sec:extensions}

There are several ways to further generalize the Gehring lemma with tails that follow by the argument used in the proof of Theorem \ref{GeneralGehering.cor}. For the sake of clear exposition, we have not included them in the main theorem, but we briefly discuss some of them in this separate section. For simplicity we work in metric spaces. We leave the adaptations to the quasi-metric setting for the interested reader. They will not be needed in the further course.

\subsection{Sequences}
\label{sec:sequences}
We usually asked the sequence $\alpha_k$ in the definition of $a_u$ to be non-increasing. Of course, this assumption can always be relaxed by asking the sequence to be non-increasing starting from a certain index $k_0$ and then replacing the terms $\alpha_k$ with {$0 \leq k \leq k_0$ with $\alpha_k' := \max_{0 \leq k \leq k_0} \alpha_k$. The resulting sequence with $\alpha_k' := \alpha_k $ for $k > k_0$} is always non-increasing and summable. 

\subsection{Maximal function}\label{sec:maximal}
The functional $a_u$ can also take the form
\[m_u^{\Omega,loc}(B(x,t)) = \sup _{ r  \in [t , (1/2) {\rm dist} \,(x, \Omega^{c}))} \fint_{B(x,r)} u \, d \mu  \]
where $\Omega \subset X$ is an open set. In other words, the supremum is over ``large" balls $B$ so that $2B \subset \Omega$. We also define  
\[m_u^{\Omega}(B(x,t)) = \sup _{ r  \in [t ,(3/4) {\rm dist} \,(x, \Omega^{c}))} \fint_{B(x,r)} u \, d \mu. \]

\begin{corollary}
\label{cor:extsec2}
Let $\Omega \subset X$ be an open set in a metric space $(X,d,\mu)$ {with doubling measure. Let $s, \beta >0$ and $q>1$ be such that $s<q$ and $\beta \geq D(1/s-1/q)$ where $D$ is any number satisfying \eqref{eq:homdim}}. Suppose that $u, f, h \ge 0$ with $u^q, f^q, h^s \in L^1_{loc}(\Omega, d\mu)$ and $A\ge 0$ is a constant such that for every ball $B = B(x, R)$ with $2 B \subset \Omega$ 
\begin{equation} 
\label{hyp:extsec2}
\left(\fint_B u^q \, d\mu\right)^{1/q} \le A m_u^{\Omega,loc}(B) + (m_{f^{q}}^{\Omega,loc} (B))^{1/q} + R^\beta(m_{h^{s}}^{\Omega,loc} (B))^{1/s}.
\end{equation} 
Then there exists $p > q$ such that for all balls $B$ with $12 B \subset \Omega$, 
\begin{equation} 
\begin{split}
\label{concl:extsec2}
\left(\fint_B u^p \, d\mu\right)^{1/p} &\lesssim m_u^{\Omega}   (  B) + (m_{f^{q}}^{\Omega} ( B))^{1/q} + R^\beta(m_{h^{s}}^{\Omega}  ( B))^{1/s}
\\& \qquad+ \left(\fint_{2B}f^p \, d\mu\right)^{1/p} + R^\beta\left(\fint_{2B}h^{ps/q} \, d\mu\right)^{q/sp},
\end{split}
\end{equation}
{where both $p$ and the implicit constant depend on $A, D, s, q, \beta$}.
\end{corollary}

\begin{proof}
We prove the claim for $B = B(x_0,R)$ with $x_{0} \in X$, $R> 0$ and {$12B \subset \Omega$}. We point out the relevant changes to the proof of Theorem \ref{GeneralGehering.cor}. Having fixed $B$, we repeat \emph{Step 1} as before (we take $N = 2$) to define $g^q= A_{R}^{q} h^s \mathbbm{1}_{2B}$ with $A_{R}$ a constant so that for any ball $B_{r}$ contained in $2B$
we have 
\[r^\beta \left(\fint_{B_{r}}h^s \, d\mu\right)^{1/s}  \lesssim    \left(\fint_{B_{r}}g^q \, d\mu\right)^{1/q} \lesssim  R^\beta \left(\fint_{2 B}h^s \right)^{1/s} . \]

In \emph{Step 2}, fix $r_0$ and $\rho_0$ real numbers satisfying $R \le r_0 <  \rho_0 \le 2R$. For $x \in B_{r_0}:= B(x_0, r_0)$, we have that 
$$B(x,  \rho_0 - r_0  ) \subset B(x_0, 2  R) \subset B(x , 4  R ),$$
and consequently for any positive function $v$,
\begin{equation}
\label{maximal1}
\begin{split}
\fint_{B(x, ( \rho_0 - r_0 ))} v \, d\mu  & \le  \frac{\mu(B(x_0, 2  R))}{\mu(B(x,\rho_0 - r_0))}   \fint_{B(x_0, 2  R)} v \, d\mu
\\& \lesssim  \left( \frac{R}{\rho_0 - r_0} \right)^{D} \fint_{B(x_0, 2 R)} v \, d\mu,   
\end{split}
\end{equation}
where we used the constant $D$ from the doubling dimension in the last line. Set 
$ \gamma  := ( R /(\rho_0 - r_0))^{D}$. 

We repeat \emph{Step 3} as it is. In \emph{Step 4}, we define three functions 
\begin{equation*}
\begin{split}
U(x,r): = \fint_{B(x,r)} u \, d \mu, \ \ \  F(x,r): = \left(\fint_{B(x,r)} f^q \, d \mu\right)^{1/q},  G(x,r): = \left(\fint_{B(x,r)} g^q \, d \mu\right)^{1/q},
\end{split}
\end{equation*}
and for $\lambda > \lambda_0$, we denote the relevant level sets by
\begin{align*}
U_\lambda := B_{r_0} \cap \{u > \lambda\}, \quad
F_\lambda := B_{r_0} \cap \{f > \lambda\}, \quad
G_\lambda := B_{r_0} \cap \{g > \lambda\}.
\end{align*}
We set
\[  \lambda_0 := C \gamma  m^\om_u(2 B) + C \left( \gamma  m^\om_{f^q}(2B)\right)^{1/q}  +  C (2R)^\beta  \left(  \gamma  m^\om_{h^s}(2B)\right)^{1/s},  \]
where $C$ is a constant independent of $u$ and the ball $B$, {chosen such that, by} an inclusion relation as in \eqref{maximal1} we obtain  
\begin{equation}
\label{maximal2}
U(x,\rho_0 - r_0) + F(x,\rho_0 - r_0) + G(x,\rho_0 - r_0) \leq \lambda_0 
\end{equation}
for all $x \in B(x_0,\rho_0)$. Finally, we define as before $$\Omega_\lambda : = \Big \{x \in U_\lambda \cup F_\lambda \cup G_\lambda: \text{$x$ is a Lebesgue point for $u, f^q$ and $g^q$} \Big\}.$$ 

In \emph{Step 5}, we note, as before, that if $x \in \Omega_\lambda$ then
$$\lim_{r \to 0} U(x,r) + F(x,r) + G(x,r) > \lambda,$$
and thus for $x \in \Omega_\lambda$ we can define the stopping time radius, this time continuously, as
\begin{equation*}
\begin{split}
r_x := \sup \big\{ r < \rho_0 - r_0 : 
 \   U(x,r) + F(x,r) + G(x,r) > \lambda \big\}.
\end{split}
\end{equation*}
Remark that \eqref{maximal2} implies that $r_{x}< \rho_0 - r_0$. 
Of course $\Omega_\lambda \subset \cup_{x \in \Omega_\lambda} B(x, r_x/5)$. By the Vitali Covering Lemma (5r-Covering Lemma) there exists a countable collection of balls $\{B(x_i,r_i)\} = \{B_i\}$ with $r_i = r_{x_i}$ such that $\{\tfrac{1}{5}B_i\}$ are pairwise disjoint and $\G_\lambda \subset \cup_i B_i$.   

We make three observations:
\begin{enumerate}
\item[(i)] For each $i$, either $\fint_{B_i} u \, d \mu \ge \frac \lambda 3$, $(\fint_{B_i} f^q \, d \mu)^{1/q} \ge  \frac \lambda 3$, or $(\fint_{B_i} g^q \, d \mu)^{1/q} \ge  \frac  \lambda 3$.
\item[(ii)] The radius of each $B_i$ is less than $\rho_0 - r_0$ and $x_i \in B(x_0,r_0)$ so $B_i \subset B(x_0,\rho_0)$.
\item[(iii)] Each $r \in [r_i, \rho_0 - r_0)$ is `above' or at the stopping time and 
\[
\fint_{B(x_i,r)} u \, d\mu + \left(\fint_{B(x_i,r)} f^q \, d\mu \right)^{1/q} +  \left(\fint_{B(x_i,r)} g^q \, d\mu \right)^{1/q} \lesssim \lambda.
\]
\end{enumerate}
We obtain from \eqref{hyp:extsec2} that
\begin{equation*}
\begin{split}
u^q(U_\lambda) &\le u^{q}(U_\lambda \cup F_\lambda \cup G_\lambda) \le \sum_i  u^q(B_{i})
\\& \lesssim \sum_i \mu(B_i)\big(m_u^{\Omega,loc}(B_i) + (m_{f^q}^{\Omega,loc}(B_i))^{1/q} + r_{i}^\beta (m_{h^s}^{\Omega,loc}(B_i))^{1/s}\big)^q.
\end{split}
\end{equation*} 
We handle the term involving $f$, to begin we split $m_{f^q}^{\Omega,loc}(B_i)$ as
\begin{equation}
\label{eq:splittingthemax}
\begin{split}
m_{f^q}^{\Omega,loc}(B_i) &= \sup _{ r  \in [r_i , \frac{1}{2} {\rm dist} \,(x_i, \Omega^{c}))} \fint_{B(x_i,r)} f^{q} \, d \mu \\
		&= \max \left( \sup _{ r  \in [r_i ,  \rho_0 - r_0)} \fint_{B(x_i,r)} f^{q} \, d \mu , \sup _{ r  \in [\rho_0 - r_0, \frac{1}{2} {\rm dist} \,(x_i, \Omega^{c}))} \fint_{B(x_i,r)} f^{q} \, d \mu \right) .
\end{split}
\end{equation}
By observation (iii), we see that 
\[\sup _{ r  \in [r_i ,  \rho_0 - r_0)} \fint_{B(x_i,r)} f^{q} \, d \mu \lesssim \lambda.\]
On the other hand,
\begin{align*}
&\sup _{ r  \in [\rho_0 - r_0, \frac{1}{2} {\rm dist} \,(x_i, \Omega^{c}))} \fint_{B(x_i,r)} f^{q} \, d \mu 
	 = \sup _{ k \in [1, \frac{1}{2(\rho_0 - r_0)} {\rm dist} \,(x_i, \Omega^{c}) )} \fint_{B(x_i, k (\rho_0 - r_0))} f^{q} \, d \mu  \\
	& \lesssim  \sup _{ k \in [1, \frac{1}{2(\rho_0 - r_0)} {\rm dist} \,(x_i, \Omega^{c}) )} \left( \frac{k (\rho_0 - r_0) + r_0}{k (\rho_0 - r_0)} \right)^{D} \fint_{B(x_0,k (\rho_0 - r_0) + r_0)} f^{q} \, d\mu  \\
	& \lesssim  \left( \frac{R}{\rho_0 - r_0} \right)^{D} m_{f^{q}}^{ \Omega}(B) \lesssim \lambda_0^{q} < \lambda^q,
\end{align*}
where the last line is justified as follows:
By the upper bound on $k$ in the supremum, we always have  
\[R \leq k (\rho_0 - r_0) + r_0 \leq  \frac{1}{2} {\rm dist} \,(x_i, \Omega^{c}) + r_0 \leq  \frac{1}{2} {\rm dist} \,(x_0, \Omega^{c}) + 3 R < \frac{3}{4} {\rm dist} \,(x_0, \Omega^{c}),\]
where we used $|x_i - x| < r_0 < \rho _0 \leq 2R$ and $B(x_0, 12 R) \subset \Omega$. Hence every $B(x_0,k (\rho_0 - r_0) + r_0)$ is admissible in the definition of $m_{f^{q}}^{ \Omega}(B)$ and we get the bound claimed before since $m_{f^{q}}^{ \Omega}(B) \leq 2^D m_{f^{q}}^{ \Omega}(2B)$. Altogether,
\begin{align*}
 (m_{f^q}^{\Omega,loc}(B_i))^{1/q} \lesssim \lambda.
\end{align*}
Terms with $u$ and $h$ are estimated similarly.
The rest of \emph{Step 5} follows as before and we obtain
\[\cup_i  B_i \subset  \{ M(u\mathbbm{1}_{B_{\rho_0}}) > \lambda/3\}\cup  \{  M(f^q\mathbbm{1}_{B_{\rho_0}}) > (\lambda/3)^q\}
 \cup \{ M(g^q\mathbbm{1}_{B_{\rho_0}}) > (\lambda/3)^q\}.\] 

\emph{Step 6} involving maximal function arguments to estimate the measure of the set in the above display {for $\lambda > \lambda_0$ as well as the overall contribution for $\lambda < \lambda_0$} is repeated without changes. In the end, we reach an inequality of the form \eqref{rdy2iterate.eq}. Indeed, set 
\[\varphi(t) := \int_{B(x_0,t)} u_m^{p-q}u^q \, d\mu, \quad \alpha_p :=  \big(m_u^{\Omega}(2B) + (m_{f^q}^{\Omega}(2B))^{1/q} + (R)^{\beta }(m_{h^s}^{\Omega}(2B))^{1/s}\big)^p \] 
and for $p \in (q,2q)$ we may summarize our estimates as 
\begin{equation*}
\varphi(r_0) \lesssim  \mu(B) \left( \frac{R}{\rho_0 - r_0} \right)^{\eta}  \alpha_p  + \epsilon_p \varphi(\rho_0) + \eps_p^{-1} \int_{2B} f^p\, d\mu + \eps_p^{-1} \int_{2B} g^p \, d\mu,
\end{equation*}
whenever $R\le r_0 < \rho_0 \le 2R$. Here $\epsilon_p = p-q$ and $\eta > 0$ is independent of $u$ and $B$. The claim \eqref{concl:extsec2} then follows from a well known iteration argument (see e.g. Lemma 6.1 in \cite{giustibook}) or from modifying the argument in \emph{Step 7}.  
\end{proof}

Note that the proof for the maximal-function-like object $m_u^{\Omega,loc}$ is actually simpler than for the tailed $a_u$. Several choices of how to discretize the scale parameters can be omitted. This setup is also very close but not comparable to Gehring's original assumption
\[(M u^{q})^{1/q} \lesssim Mu\]
where $q > 1$ and $M$ the Hardy--Littlewood maximal operator. Indeed, the left-hand side here does not have a maximal function and the right-hand side is a maximal function restricted to large scales (a non-local maximal function). 

\subsection{Domains}

We can define the tail functional $a_u^{\Omega,loc}$ restricted to an open set $\Omega$, for example
\[a_u^{\Omega,loc}(B) := \sum_{ \substack{k \geq 0 \\
									 2^{k+4} B \subset \Omega} } \alpha_k \fint_{2^{k} B} u \, d \mu   \]
and 
\[a_u^{\Omega}(B) := \sum_{ \substack{k \geq 0 \\
									 2^{k+1} B \subset \Omega} } \alpha_k \fint_{2^{k} B} u \, d \mu, \]
{where as before $(\alpha_k)_k$ is a non-increasing and summable sequence of positive numbers.} Then we can localize the assumptions of Theorem \ref{GeneralGehering.cor} to $\Omega$. 

\begin{corollary}
\label{cor:extsec_1}
Let $\Omega \subset X$ be an open set in a metric space $(X,d,\mu)$ {with doubling measure. Let $s, \beta >0$ and $q>1$ be such that $s<q$ and $\beta \geq D(1/s-1/q)$ where $D$ is any number satisfying \eqref{eq:homdim}}. Suppose that $u, f, h \ge 0$ with $u^q, f^q, h^s \in L^1_{loc}(\Omega, d\mu)$ and  $A\ge 0$ is a constant  such that for every ball $B = B(x, R)$ with $16 B \subset \Omega$ 
\begin{equation} 
\label{hyp:extsec_1}
\left(\fint_B u^q \, d\mu\right)^{1/q} \le Aa_u^{\Omega,loc}(B) + (a_{f^{q}}^{\Omega,loc}(B))^{1/q} + R^\beta(a_{h^{s}}^{\Omega,loc}(B))^{1/s}.
\end{equation} 
Then there exists $p>q$ such that for all balls $B$ with $32 B \subset \Omega$, 
\begin{equation} 
\begin{split}
\left(\fint_B u^p \, d\mu\right)^{1/p} &\lesssim a_u^{\Omega }( 4B) + (a_{f^{q}}^{\Omega}( 4B))^{1/q} + R^\beta(a_{h^{s}}^{\Omega}(4B))^{1/s}
\\& \qquad+ \left(\fint_{4 B}f^p \, d\mu\right)^{1/p} + R^\beta\left(\fint_{4 B}h^{ps/q} \, d\mu\right)^{q/sp},
\end{split}
\end{equation}
{where both $p$ and the implicit constant depend on $A, \beta, s,q, D$.}
\end{corollary}

\begin{proof}
Theorem \ref{GeneralGehering.cor} shows how to deal with the tail. Corollary \ref{cor:extsec2} shows how to adapt the proof to the setting relative to $\Omega$. The proof of this Corollary can be reconstructed following the proof of Theorem \ref{GeneralGehering.cor} and carefully adapting the estimation in \eqref{eq:splittingthesum} in the spirit of estimating \eqref{eq:splittingthemax} to make sure that all relevant balls appearing in the estimates are contained in $\Omega$. 
\end{proof}

\subsection{Convolutions}
In the Euclidean setting where $(X,d,\mu)$ is $\mathbb{R}^{n}$ equipped with the usual distance and the Lebesgue measure, we can realize the functionals $a_u$ as convolutions
\[ a_u(B(x,r)) =  (\varphi_r * u)(x) \]
where $\varphi$ has suitable decay and integrability and $\varphi_{r}(x)= r^{-n}\varphi(x/r)$. More precisely, our assumptions correspond to $\varphi$ being bounded, radial, decreasing and globally integrable. A convolution makes sense in certain groups, so this kind of special functional can also be considered, for instance, in nilpotent Lie groups as in \cite{VSC}. 

\section{Very weak \texorpdfstring{$A_{\infty}$}{Aoo} weights}
\label{sec:Ainfty}

For a weight (that is, a non-negative locally integrable function), the condition 
$$
\int_{B} M(\mathbbm{1}_{B}w) \, d\mu \le C  \int_{B} w\, d\mu 
$$
valid for some $C<\infty$ and all balls $B$ of $X$ can be taken as a definition of the $A_{\infty}$ class, where $M$ is the uncentered maximal operator, see \cite{Fujii, W1987} for the Euclidean case with Lebesgue measure. In spaces of homogeneous type, this condition implies higher integrability with an exponent that can be computed from the constant $C$ and the structural constants of $X$, see \cite{HPR2012}. This was extended in \cite{AHT2017} to  weights in the weak $A_{\infty}$ class defined by 
\begin{equation}
\label{eq:vws}
\fint_{B} M(\mathbbm{1}_{B}w) \, d\mu \le C   \fint_{\sigma B} w\, d\mu.
\end{equation}
where $\sigma>1$ is given. The classes are shown to be independent of $\sigma$ provided $\sigma>K$, $K$ being the quasi-metric constant, and their elements still have a higher integrability. The methods passing through a dyadic analog  yield an accurate estimate of the exponent in terms of the best $C$ in the definition.  We note  that the dilation parameter $\sigma$ is uniform: it is the same for all balls. Our methods allow us to  remove the uniformity, that is we define the \emph{very weak $A_{\infty}$ class} as the set of weights such that for all balls, 
\begin{equation}
\label{eq:vw}
\fint_{B} M(\mathbbm{1}_{B}w) \, d\mu \le C  \sup_{\sigma\ge 1} \fint_{\sigma B} w\, d\mu < \infty.
\end{equation}
The quantity in the middle is the same functional as the one defined in Section \ref{sec:maximal} when $\Omega=X$. 

We denote by \vw this class. As the right-hand side of \eqref{eq:vw} requires boundedness of all averages on large balls, this rules out weights growing at $\infty$.  For this reason,   it is neither contained in, nor containing  the class $A_{\infty}^{\mathrm{weak}}$ introduced in \cite{AHT2017}. Typically, such very weak  $A_{\infty}$ weights arise from fractional equations. See the next section. 

\begin{theorem}\label{thm:ainfty} For any very weak $A_{\infty}$ weight $w$, there exist $p>1$ and $C'<\infty$ such that for all balls $B$, 
\begin{equation}
\label{eq:vwp}
\bigg(\fint_{B} M(\mathbbm{1}_{B}w)^p \, d\mu\bigg)^{1/p}  \le C'  \sup_{\sigma\ge 1} \fint_{\sigma B} w\, d\mu.
\end{equation}
\end{theorem}
 
\begin{remark}
\label{remark:weights}
The improvement of integrability on a given ball $B$ only depends on the finiteness of the right-hand side  for that same ball and nothing else, as the proof will show.  Hence, one can also define the \emph{very weak $A_{\infty}$ class on $B$} by the condition \eqref{eq:vw} on that very ball. The theorem remains valid if one replaces (in the assumption and the conclusion) the supremum by a tail as before. {That variant leads to the class $C_p$ (see Section \ref{sec:cp}).} The advantage is to allow some  possible growth for which the tail is finite while the supremum is not. Finally our argument works with the supremum replaced by one average  with a fixed dilation parameter. {We leave these extensions to the interested reader. They will not be needed here.}
\end{remark}

\begin{proof} 
To simplify we do the proof in the metric case. Again the trick to reduce the quasi-metric case to the metric case applies, see Section~\ref{sec:quasi}. The argument follows again that of Theorem \ref{GeneralGehering.cor} with $f,h=0$ but with some changes.  

We pick $N=2$. We ignore \emph{Step~1} and have the setup of \emph{Step~2}. Having fixed the ball $B=B(x_{0}, R)$, the parameters $ \rho_{0}, r_{0}$, and $\ell$ such that $2^\ell(\rho_{0}-r_{0})=R$, define
$$\widetilde M v(x) := \sup_{k\in \Z} \fint_{B(x, 2^k(\rho_{0}-r_{0}))} |v|\, d\mu.$$
Then if $M_{c}$ designates the centered maximal operator, 
$$\widetilde M v \le M_{c}v \le Mv \le \kappa'M_{c}v \le \kappa \widetilde Mv.$$
Indeed, $Mv \le \kappa'M_{c}v$ is classical, while $M_{c}v\lesssim \widetilde Mv$ follows from the doubling property and $\kappa$ does not depend on $\rho_{0}-r_{0}$ in particular.  By the same token,   in the right-hand side of \eqref{eq:vw}   we may restrict to the  supremum over all $\sigma= 2^k$ for  integers $k\ge 0$. This only causes a change in the constant $C$.

We modify \emph{Step~3} as follows. With the truncation of the maximal function at level $m$, 
\begin{equation}
\label{Ainfty.eqstartlayercake.eq}
\begin{split}
\int_{B_{r_0}} (M(\mathbbm{1}_{B_{r_{0}}}w))_m^{p} \, d\mu &  \le \kappa^{p} \int_{B_{r_0}} (\widetilde M(\mathbbm{1}_{B_{r_{0}}}w))_{m/\kappa}^{p-1}\ \widetilde M(\mathbbm{1}_{B_{r_{0}}}w) \, d\mu
\\
&
\le  \kappa^{p} \int_{B_{r_0}} (\widetilde M(\mathbbm{1}_{B_{\rho_{0}}}w))_{m/\kappa}^{p-1}\ \widetilde M(\mathbbm{1}_{B_{\rho_{0}}}w) \, d\mu
\\
&
= \kappa^{p} (p-1)\int_0^{m/\kappa} \lambda^{p - 2} u(B_{r_0} \cap \{u > \lambda\})\, d\lambda
\end{split}
\end{equation}
with $u:=\widetilde M(\mathbbm{1}_{B_{\rho_{0}}}w).$

In \emph{Step~4}, we pick $\lambda_{0} := C_d^{\ell +2} \sup_{\sigma\ge 1} \fint_{\sigma B} w\, d\mu $ (which is assumed finite otherwise there was nothing to prove), where we recall that $C_{d}$ is the doubling constant. We observe that for $x\in B_{r_{0}} = B(x_0,r_0)$ and $k \ge 0$,
\begin{align}
\label{eq1:Ainfty}
 \fint_{B(x,2^{k}(\rho_{0}-r_{0}))} \mathbbm{1}_{B_{\rho_{0}}} w \, d\mu  \le C_d^{\ell +2}  \fint_{B(x_0,2^{k+1}R)} w \, d\mu \le \lambda_{0}.
\end{align}

The stopping time of \emph{Step~5} is slightly different. Let $\lambda  > \lambda_0$. Pick $x\in B_{r_0} \cap \{u > \lambda\}$. As $u(x)=\widetilde M(\mathbbm{1}_{B_{\rho_{0}}}w)(x)>\lambda>\lambda_{0}$, the observation above and $B(x,2^{k}(\rho_{0}-r_{0}))\subset B_{\rho_{0}}$ when $k<0$ imply 
$$u(x) =\sup_{k<0}  \fint_{B(x,2^{k}(\rho_{0}-r_{0}))} \mathbbm{1}_{B_{\rho_{0}}} w \, d\mu= \sup_{k<0}  \fint_{B(x,2^{k}(\rho_{0}-r_{0}))}  w \, d\mu .$$  Let $k_{x}<0$ be the supremum of those $k<0$ for which $\fint_{B(x,2^{k}(\rho_{0}-r_{0}))}  w \, d\mu >\lambda$. We extract the covering $B_{i}=B(x_{i}, 2^{k_{x_{i}}}(\rho_{0}-r_{0}))$ of $B_{r_0} \cap \{u > \lambda\}$, where all $B_{i}$ are subballs of $B_{\rho_{0}}$ {with the $\frac{1}{5} B_i$ pairwise disjoint. We claim} that  if  $B_{i}^*=2B_{i}$, then for all $x\in B_{i}\cap B_{r_{0}}$, we have $u(x)\le C_{d}^2 M(\mathbbm{1}_{B_{i}^*} w)(x)$.

Indeed, fix $x\in B_{i}\cap B_{r_{0}}$ and pick $k\in \Z$. In the case where $k\ge 0$ we have by \eqref{eq1:Ainfty},
$$
\fint_{B(x,2^{k}(\rho_{0}-r_{0}))} \mathbbm{1}_{B_{\rho_{0}}} w \, d\mu   \le \lambda_{0} <\lambda  < \fint_{B_{i}}  w \, d\mu \le  M(\mathbbm{1}_{B_{i}} w)(x).
$$
 In the  case where $0>k\ge k_{x_{i}}$, we have {either by the stopping time or again by \eqref{eq1:Ainfty} if $k=-1$,}
\begin{align*}
\label{}
   \fint_{B(x,2^{k}(\rho_{0}-r_{0}))}  \mathbbm{1}_{B_{\rho_{0}}} w \, d\mu &\le  \frac{\mu(B(x_{i},2^{k+1}(\rho_{0}-r_{0})))}{\mu(B(x,2^{k}(\rho_{0}-r_{0})))}\fint_{B(x_{i},2^{k+1}(\rho_{0}-r_{0}))}  w \, d\mu \\& \le C_{d}^2 \lambda < C_{d}^2  \fint_{B_{i}}  w \, d\mu \le C_{d}^2 M(\mathbbm{1}_{B_{i}} w)(x).    \end{align*}
In the  case where $k<k_{x_{i}}$,  $B(x,2^{k}(\rho_{0}-r_{0}))\subset B_{i}^*$ and $B(x,2^{k}(\rho_{0}-r_{0})) \subset B_{\rho_{0}}$, hence 
$$
 \fint_{B(x,2^{k}(\rho_{0}-r_{0}))}  \mathbbm{1}_{B_{\rho_{0}}} w \, d\mu=  \fint_{B(x,2^{k}(\rho_{0}-r_{0}))}  w \, d\mu \le  M(\mathbbm{1}_{B_{i}^*} w)(x).
$$
Thus, the intermediate claim is proved.

Now, using this together with \eqref{eq:vw} and the opening remark, we obtain
 \begin{align*}
\label{}
   u(B_{r_0} \cap \{u > \lambda\})\le \sum_{i} u(B_{i}\cap B_{r_{0}}) &\le  \sum_{i}   \int_{B_{i}\cap B_{r_{0}}} u\, d\mu  \\&\le \sum_{i}  C_{d}^2 \int_{B_{i}\cap B_{r_{0}}} M(\mathbbm{1}_{B_{i}^*} w)\, d\mu   \\
   &\le \sum_{i}  C_{d}^2\   \mu(B_{i}^*) \fint_{B_{i}^*} M(\mathbbm{1}_{B_{i}^*} w)\, d\mu 
   \\&
   \le C C_{d}^2 \sum_{i}  \mu(B_{i}^*) \sup_{k\ge 0 } \fint_{2^k B_{i}^*} w\, d\mu \\&
   \le C C_{d}^2 \sum_{i}  \mu(B_{i}^*)\  \lambda
   \\
   &  \lesssim \lambda\  \mu(\cup B_{i}). 
\end{align*}
The next to last inequality is by definition of $B_{i}^*=2B_{i}$, hence all the averages do not exceed {$\lambda_0 <\lambda$ by \eqref{eq1:Ainfty}},  and the last inequality uses doubling and the fact that $\frac{1}{5}B_{i}$ are disjoint.  As $B_{i}\subset B_{\rho_{0}}$ and  $\lambda<  \fint_{B_{i}}  w \, d\mu $
we have $ \cup B_{i} \subset B_{\rho_{0}}\cap \{  M(\mathbbm{1}_{B_{\rho_{0}}} w)>\lambda\}
$
and we have obtained
$$
u(B_{r_0} \cap \{u > \lambda\}) \lesssim \lambda\  \mu(B_{\rho_{0}}\cap \{  M(\mathbbm{1}_{B_{\rho_{0}}} w)>\lambda\}).
$$

\emph{Step~6} is now done as follows by cutting the rightmost integral in \eqref{Ainfty.eqstartlayercake.eq} at $\lambda_{0}$. Let $\varphi(r_{0}) :=  \int_{B_{r_0}} (M(\mathbbm{1}_{B_{r_{0}}}w))_m^{p} \, d\mu$. Then 
\begin{align*}
\label{}
\varphi(r_{0})    & \le  \kappa^{p} \lambda_0^{p-1}   \ u(B_{r_0})  +  \kappa^{p} (p-1)  \int_{\lambda_{0}}^{m/\kappa} \lambda^{p - 2} u(B_{r_0} \cap \{u > \lambda\})\, d\lambda  \\
&   \lesssim  \mu(B)C_{d}^{p\ell}  \ \bigg(\sup_{\sigma\ge 1} \fint_{\sigma B} w\, d\mu\bigg)^p  + 
     (p-1)  \int_{\lambda_{0}}^{m/\kappa} \lambda^{p - 1} \mu(B_{\rho_{0}}\cap \{  M(\mathbbm{1}_{B_{\rho_{0}}} w)>\lambda\})\, d\lambda
    \\
& \lesssim \mu(B)C_{d}^{p\ell}  \ \bigg(\sup_{\sigma\ge 1} \fint_{\sigma B} w\, d\mu\bigg)^p+   \, \frac{p-1}{p} \int_{B_{\rho_{0}}}  M(\mathbbm{1}_{B_{\rho_{0}}} w)^p_{m/\kappa}\, d\mu. 
\end{align*}
We recall that $\ell$ was defined by $2^\ell(\rho_{0}-r_{0})= R$. As $\kappa>1$, we have obtained $$ \varphi(r_{0}) \lesssim  \mu(B)C_{d}^{p\ell}    \ (\sup_{\sigma\ge 1} \fint_{\sigma B} w\, d\mu)^p + \epsilon_{p}\varphi(\rho_{0}).$$ From there, we do as in \emph{Step~7} an iteration provided $p-1$ is small and finally let $m\to \infty$ to deduce \eqref{eq:vwp}. 
\end{proof}

Having this theorem at hand, we can proceed as in \cite{AHT2017} and show the equality of the class \vw with 
other classes.    We say that a weight is a  \emph{very weak $\mathcal{A}_{\infty}$  weight},    if there exist an exponent  $1<p<\infty$ and a constant $C$ such that   for all balls $B$ and Borel subsets $E$ of $B$, 
\begin{equation}
\label{eq:vwAp}
 0 <  \inf_{\sigma\ge 1} \frac{w(E)}{w(\sigma B)}\frac{\mu(\sigma B)}{\mu(B)} \le C \bigg(\frac{\mu(E)}{\mu(B)}\bigg)^{1/p}.
\end{equation}
We call \vwp this class. 
We say that a weight $w$ is a \emph{very weak reverse H\"older weight} if there exist an exponent  $1<q<\infty$ and a constant  $C<\infty$ such that   for all balls $B$,
\begin{equation}
\label{eq:rhp}
\bigg(\fint_{B} w^q \, d\mu\bigg)^{1/q}  \le C  \sup_{\sigma\ge 1} \fint_{\sigma B} w\, d\mu < \infty.
\end{equation}
We call \rhq this class.

\begin{theorem} Let $w$ be a weight and $B$ be a ball of $X$. The  condition \eqref{eq:vw},  \eqref{eq:vwAp} for some $p\in (1,\infty)$ and \eqref{eq:rhp} for some $q\in (1,\infty)$ are equivalent (with different constants). In particular, we have
coincidence of \vw,  \vwp and \rhq. 
 \end{theorem}

\begin{proof}
Adapt the proof of Lemma 8.2 in \cite{AHT2017} together with our Theorem~\ref{thm:ainfty} {as the proper replacement for Theorem~5.6 therein}. 
\end{proof}

\subsection{\texorpdfstring{$C_p$}{Cp} weights}
\label{sec:cp}

Let $X = \mathbb{R}^{n}$ equipped with Euclidean distance and Lebesgue measure and write $|E|$ for the Lebesgue measure of a set $E$. Fix a weight $w$. Upon replacing the supremum $\sup_{\sigma > 1} w(\sigma B)/|\sigma B|$ by the tail functional
\begin{align*}
a_{C_p}(B) &:= \frac{1}{|B|} \int_{\mathbb{R}^{n}} M( \mathbbm{1}_{B})^{p} w\, dx \eqsim \fint_{B} w \, dx +  \frac{1}{|B|}\sum_{k = 1}^{\infty}  \int_{2^{k + 1}B \setminus 2^{k} B}  \left( \frac{|B|}{|2^{k} B| } \right)^{p} w\, dx  \\
&\eqsim \sum_{k=1}^{\infty} 2^{-kn(p-1)} \fint_{2^{k}B} w \, dx  
\end{align*}
with $ 1<p<\infty$ in the definition of \vwp in \eqref{eq:vwp}, we recover the $C_p$ condition of Muckenhoupt \cite{Muckenhoupt1981} and Sawyer \cite{Sawyer1983}. Namely, we say that $w \in C_p$ if there are $\delta > 0$ and $C > 0$ so that
\begin{equation}
\label{eq:cp}
w(E)  \leq  C \left( \frac{| E |}{|B|} \right)^{\delta} \int_{\mathbb{R}^{n}}  M( \mathbbm{1}_{B})^{p} w\, dx < \infty
\end{equation}
holds for all balls $B$ and measurable $E \subset B$. 

Following the proof of Lemma 8.2 in \cite{AHT2017}, we see that $w \in C_p$ if and only if
there are $\delta' > 0$ and $C > 0$ such that  for all balls $B$,
\[ \left( \fint_{B} w^{1+ \delta'} \, dx \right)^{1/(1+\delta')} \leq C a_{C_p}(B) < \infty.\]
 Modifying the proof of Theorem \ref{thm:ainfty} (see Remark \ref{remark:weights}), one can append 
\[\fint_{B} M(\mathbbm{1}_{B}w) \, dx \leq C a_{C_p}(B) < \infty \]
holding for some $C > 0$ and all balls $B$ to the list of equivalent definitions of the $C_{p}$ class. In conclusion, the class $C_p$ gives examples of functions satisfying a reverse H\"older inequality with tail as in Theorem \ref{GeneralGehering.cor}. Conversely, as we prove next a reverse H\"older inequality with a tail of the form $a_{C_p}(B)$ for fractional derivatives of solutions to fractional divergence form equations, we see that solutions produce examples of $C_p$ weights.

\section{Application to a fractional divergence form equation}\label{sec:appl} 

Throughout this section let $\alpha \in (0,1)$. We illustrate our main results by applying them to solutions $u$ on $\RR^n$, $n \geq 2$, to the fractional divergence form equation of order $2\alpha$ introduced by Shieh--Spector~\cite{ShiehFractional}. This equation is formally given by
\begin{align}
\label{eq:fractional_equation}
 (D^\alpha)^* (A D^\alpha u) = (D^\alpha)^* F + f,
\end{align}
where $D^\alpha$ is the \emph{Riesz fractional gradient} defined for $u \in L^2(\RR^n)$ as the $\CC^n$-valued tempered distribution with $\RR^n$-valued Fourier symbol $\xi/|\xi|^{1-\alpha}$, that is to say, for $\xi \in \RR^n$ we have
\begin{align*}
 \mathcal{F} (D^\alpha u)(\xi) = \frac{\xi}{|\xi|^{1-\alpha}} \mathcal{F} u(\xi).
\end{align*}
Throughout, we use the normalization 
\begin{align*}
 \mathcal{F}u(\xi) := \frac{1}{(2 \pi)^{n/2}} \int_{\mathbb{R}^{n}} u(x) e^{- i \xi \cdot x} \, dx
\end{align*}
for the Fourier transform. Note that for $\alpha = 1$ we would recover the classical divergence form structure with $D^1 = \nabla$ and adjoint $(D^1)^* = - \mathrm{div}$. Also note that $F$ in \eqref{eq:fractional_equation} is $\CC^n$-valued whereas $f$ is scalar valued, but since there is no danger of confusion we shall not use different notation for vector valued functions. As for the coefficients, we assume that $A: \RR^n \to \CC^{n \times n}$ is bounded, measurable, and that there exists $\lambda > 0$ such that for all $x \in \RR^n$ and all $\xi \in \CC^n$,
\begin{align}
\label{eq:ellipticity}
 \lambda |\xi|^2 \leq \Re(A(x)\xi \cdot \cl{\xi}) \leq \lambda^{-1} |\xi|^2.
\end{align}

As in \cite{ShiehFractional, SchikorraVMO}, we study \eqref{eq:fractional_equation} in a global variational framework using the Hilbert space $H^{\alpha,2} = H^{\alpha,2}(\RR^n)$ that consists of all $u \in L^2(\RR^n)$ such that $|D^\alpha u| \in L^2(\RR^n)$ and that is endowed with the natural norm $u \mapsto (\|u\|_{L^2}^2 + \|D^\alpha u\|_{L^2}^2)^{1/2}$. Since $\xi/|\xi|^{1-\alpha}$ is comparable to $|\xi|^\alpha$ is Euclidean norm, $H^{\alpha,2}$ coincides up to equivalent norms with the Bessel potential space that is usually denoted by the same symbol. Fractional Sobolev embedding theorems give us two important indices for every $p \in (1,\infty)$:
\begin{align*}
 p^{*} := \frac{pn}{n-\alpha p} \qquad \textrm{and} \qquad p_{*}  := \frac{pn}{n+ \alpha p}, 
\end{align*} 
that satisfy $1/p^{*} + 1/p_{*} = 2/p$ and $(p_*)^* = (p^*)_* = p$. The value of $\alpha$ in the definition of $p_*$ and $p^*$ will usually be clear from the context. Otherwise it will be given explicitly as $p_{*,\alpha}$ or $p^*_{\alpha}$. In particular, $H^{\alpha,2}$ embeds continuously into $L^{2^*} = (L^{2_*})'$, see Theorem~1.2.4 in \cite{Adams-Hedberg}.

\begin{definition}
\label{def:solution}
Let $F \in L^2$ and $f \in L^{2_*}$. A function $u \in H^{\alpha ,2}$ is called \emph{weak solution} to \eqref{eq:fractional_equation} if for all $\varphi \in H^{\alpha,2}$,
\begin{align*}
 \int_{\mathbb{R}^{n}} A(x) D^\alpha u(x) \cdot \cl{D^\alpha \varphi(x)}\, dx = \int_{\mathbb{R}^{n}}  F(x) \cdot \cl{D^\alpha \varphi(x)} +  f(x)\cl{\varphi(x)} \, dx.
\end{align*}
\end{definition}

Our goal is to prove that weak solutions to \eqref{eq:fractional_equation} exhibit locally higher integrability of their derivatives of order $\alpha$. There are at least two legitimate choices for defining such fractional derivatives: One is the Riesz fractional gradient $D^\alpha u$, the other one the \emph{fractional Laplacian} $(-\Delta)^{\alpha/2}u$ that we introduce for $u \in H^{\alpha,2}$ through the Fourier symbol $|\xi|^\alpha$ or, equivalently, through the singular integral representation for almost every $x\in \RR^n$,
\begin{align}
\label{eq:fractional_Laplacian}
 ( - \Delta)^{\alpha/2} u (x) = c_{n,\alpha} \int_{\RR^{n}} \frac{u(x) - u(y)}{|x-y|^{n+ \alpha} } \, dy,
\end{align}
where integral is understood in the principal value sense~\cite{TenFractionalLaplacians, Stein}. Our main result then reads as follows.

\begin{theorem}
\label{thm:fractional}
Let $u \in H^{\alpha, 2}$ and $p > 2$. Suppose $u$ is a weak solution to \eqref{eq:fractional_equation}, where $f \in L^{2_{\ast}} \cap L_{loc}^{p_{\ast}} $ and $F \in L^{2} \cap L_{loc}^{p}$. Then there exists $\epsilon_0 = \epsilon_0(\lambda, n, \alpha,p) >  0$ such that
\begin{enumerate}
 \item local higher integrability $|(- \Delta)^{\alpha/2} u| + |D^\alpha u| \in L_{loc}^{2+ \epsilon_0}$ holds,
 \item global integrability $f \in L^{2_\ast} \cap L^{p_\ast} $ and $F \in L^{2} \cap L^{p}$, implies global higher integrability $|(- \Delta)^{\alpha/2} u| + |D^\alpha u| \in L^{2+ \epsilon_0}$.
\end{enumerate}
\end{theorem}

Below, we shall prove this result by first establishing new reverse H\"older inequalities with tails for solutions to \eqref{eq:fractional_equation} and then applying the non-local Gehring lemmas from Theorem \ref{GeneralGehering.cor} and Theorem \ref{thm:global1}. As a consequence, there are in fact quantitative bounds that substantiate assertion (1) and (2) above. We shall write them out in \eqref{eq:RHbound_solution} at the end of this section.

We remark that for equations with real $VMO$-coefficients $A$ it was shown in \cite{SchikorraVMO} by different techniques that global $p$-integrability of $F$ leads to local $p$-integrability of $D^\alpha u$ for any $p>2$. The local higher integrability of $D^\alpha u$ in (1) should therefore be seen as the counterpart of that result for equations with merely measurable coefficients. Although $(-\Delta)^{\alpha/2}u$ is comparable to $D^\alpha u$ only on the global level using Fourier multipliers (the Riesz transform), we obtain local self-improvement for the former quantity, too. This, as well as the global results (2), seem to be novel for fractional divergence form equations even in the setting of \cite{Schikorra2016}. 

\subsection{Reverse H\"older estimates for the Riesz fractional gradient}

We begin with a reverse H\"older estimate for the Riesz fractional gradient of solutions. Surprisingly, this estimate is still local in $u$ and $F$.

\begin{proposition}
\label{prop:RH_RieszGradient}
Let $u \in H^{\alpha,2}$ be a weak solution to \eqref{eq:fractional_equation}, where $f \in L^{2_*}$ and $F \in L^2$. Let $\rho \in (2_{*,1},2)$. Then for all balls $B = B(x,r) \subset \RR^n$,
\begin{align*}
 \bigg(\fint_B |D^\alpha u|^2 \bigg)^{1/2}
&\lesssim \bigg(\fint_{2B} |D^\alpha u|^\rho \bigg)^{1/\rho} + \bigg(\fint_{2B} |F|^2 \bigg)^{1/2}\\
&\quad + r^\alpha \bigg(\sum_{k=0}^\infty  2^{-k(1-\alpha)} \fint_{2^k B} |f|^{2_{*,\alpha}}\bigg)^{1/2_{*,\alpha}},
\end{align*}
with an implicit constant depending on $n, \alpha, \lambda, \rho$.
\end{proposition}

\begin{remark}
\label{rem:RH_RieszGradient}
The proof below will yield in fact a stronger version of the above estimate, see in particular \eqref{eq:RH_RieszGradient_Goal}: All but the averages for $k=0,1$ of $f$ can be taken in $L^1$ instead of $L^{2_{*,\alpha}}$ and $\rho=2_{*,1}$ is admissible in all dimensions but $n =2$. For the applications we shall not need such precision.
\end{remark}

For the proof we need to recall the notions of \emph{Riesz potential} and \emph{Riesz transform}. The potential $I_s$ of order $s \in (0,1)$ corresponds to the Fourier symbol $|\xi|^{-s}$. For $g \in L^p$, $p \in (1, n/s)$, we have $I_s g \in L^{p^*_s}$, a bounded operator $I_s: L^p \to L^{p^*_s}$ and an absolutely convergent representation 
\begin{align*}
 I_s g(x) = c_s \int_{\RR^n} \frac{g(y)}{|x-y|^{n-s}} \, dy
\end{align*}
for almost every $x \in \RR^n$. For $v \in H^{s,2}$ we have $(-\Delta)^{s/2} I_s v = v = I_s (-\Delta)^{s/2} v$. We refer to Section V.1 in \cite{Stein} for all these properties. We shall also use the well known \emph{Riesz transform} $R$ and its adjoint $R^*$ corresponding to the symbols $- i \xi/|\xi|$ and $i \xi^\top/|\xi|$, respectively. For $g \in L^p$ and $G \in (L^p)^n$, $p \in (1,\infty)$, they are given for almost every $x \in \RR^n$ by the principal value integrals
\begin{align}
\label{eq:RieszTransforms}
 R g (x) = c_n \int_{\RR^n} \frac{x-y}{|x-y|^{n+1}} g(y) \, dy \quad \text{and} \quad R^* G (x) = c_n \int_{\RR^n} \frac{x-y}{|x-y|^{n+1}} \cdot G(y) \, dy,
\end{align}
see Section III.1 in \cite{Stein}. For $u \in H^{\alpha,2}$ we can read off from the respective Fourier symbols the important relations
\begin{align}
\label{eq:relations_fractional_derivatives}
D^\alpha u = \nabla I_{1-\alpha} u = R (-\Delta)^{\alpha/2} u.
\end{align}

We are now ready to give the

\begin{proof}[Proof of Proposition~\ref{prop:RH_RieszGradient}]
We let $B$ be a ball of radius $1$ and pick an adapted cut-off function $\varphi \in C_0^\infty(2B)$ satisfying $0 \leq \varphi \leq 1$, $\varphi = 1$ on $B$ and $|\nabla \varphi| \lesssim 1$. We claim that it suffices to establish the bound
\begin{align}
\label{eq:RH_RieszGradient_Goal}
\begin{split}
\int_{\RR^n} |D^\alpha u|^2 \varphi^2 
&\lesssim \eps \int_{\RR^n} |D^\alpha u|^2 \varphi^2 + \eps^{-1} \bigg(\int_{2B} |D^\alpha u|^\rho \bigg)^{2/\rho} + \eps^{-1} \int_{2B} |F|^2  \\
&\quad+ \eps^{-1}\bigg(\int_{4B} |f|^{2_*}\bigg)^{2/2_*} + \bigg(\sum_{k=2}^\infty 2^{-k(n+1-\alpha)} \int_{2^k B} |f| \bigg)^2
\end{split}
\end{align}
for $\eps>0$ and an implicit constant depending on $n,\alpha, \lambda$. Indeed, upon choosing $\eps$ sufficiently small to absorb the first term on the right into the left-hand side and using the defining properties of $\varphi$, this implies the claim for balls of radius $1$ even in its stronger form alluded to in Remark~\ref{rem:RH_RieszGradient}. In order to pass to balls of arbitrary radius $r>0$, we apply the above to $x \mapsto u(xr)$, which is a solution to an equation of the same form as \eqref{eq:fractional_equation} with \emph{identical} ellipticity constants.

In the following, implicit constants will only depend on $n,\alpha, \lambda$ and we shall not mention this in every single step. In order to prove \eqref{eq:RH_RieszGradient_Goal} we start out with ellipticity of $A$, see \eqref{eq:ellipticity}, and write
\begin{align*}
 \int_{\RR^n} |D^\alpha u|^2 \varphi^2 \leq \Re \int_{\RR^n} A D^\alpha u \cdot \cl{(D^\alpha u)} \varphi^2.
\end{align*}
We introduce the potential $v := I_{1-\alpha} u - \fint_{2B} I_{1-\alpha} u$ with zero average on $2B$. Since $u \in L^2$ we have $I_{1-\alpha} u \in L^{2^*_{1-\alpha}} \subset L^2_{loc}$ and hence $v \in L^2_{loc}$. Moreover, $\nabla v = D^\alpha u \in L^2$ due to \eqref{eq:relations_fractional_derivatives}, which leads to
\begin{align*}
 \int_{\RR^n} |D^\alpha u|^2 \varphi^2 
\leq \Re \int_{\RR^n} A D^\alpha u \cdot \cl{\nabla(v \varphi^2)} - 2  \Re \int_{\RR^n} A D^\alpha u \cdot \cl{v \varphi \nabla \varphi} =: \Re(I) + \Re(II).
\end{align*}
As for $II$, we recall boundedness of $A$, $\varphi$, $\nabla \varphi$, apply Young's inequality with $\eps$ and then use the standard Sobolev-Poincar\'e inequality for functions with zero average on a ball (Corollary~8.1.4 in \cite{Adams-Hedberg}) to give
\begin{align*}
 |II| &\lesssim \eps \int_{\RR^n} |D^\alpha u|^2 \varphi^2 + \eps^{-1} \int_{2B} |v|^2 \\
 &\lesssim \eps \int_{\RR^n} |D^\alpha u|^2 \varphi^2 + \eps^{-1} \bigg(\int_{2B} |\nabla v|^{\rho} \bigg)^{2/\rho}.
\end{align*}
Since we have $\nabla v = D^\alpha u$, this is a desirable bound in view of \eqref{eq:RH_RieszGradient_Goal}.

We turn to $I$. Since $v \varphi^2 \in H^{1,2}$, we can use the Fourier transform to write $\nabla(v \varphi^2) = D^\alpha (-\Delta)^{(1-\alpha)/2}(v \varphi^2)$, which in turn allows us to bring into play the equation for $u$ with $(-\Delta)^{(1-\alpha)/2}(v \varphi^2) \in H^{\alpha,2}$ as a test function:
\begin{align*}
I = \int_{\RR^n} F \cdot \cl{\nabla(v \varphi^2)} + \int_{\RR^n} f \, \cl{(-\Delta)^{(1-\alpha)/2}(v \varphi^2)} =: I_1 + I_2.
\end{align*}
By the support properties of $\varphi$ and the relation $D^\alpha u = \nabla v$ we have
\begin{align*}
 I_1 = \int_{2B} F \cdot \cl{(D^\alpha u)} \varphi^2 + \int_{2B} 2 F \cdot \cl{v} \varphi \nabla \varphi.
\end{align*}
A desirable bound for the first integral on the right is obtained simply from Young's inequality with $\eps$, whereas for the second one we also invoke the Sobolev-Poincar\'e inequality for $v$ as in the treatment of $II$ above. This completes the handling $I_1$ and the only term remaining is $I_2$.

Let us further split $I_2$ into a local and a global piece
\begin{align*}
 I_2 = \int_{\RR^n} \mathbbm{1}_{4B} f \, \cl{(-\Delta)^{(1-\alpha)/2}(v \varphi^2)} + \int_{\RR^n} \mathbbm{1}_{(4B)^c} f \, \cl{(-\Delta)^{(1-\alpha)/2}(v \varphi^2)} =: I_{21} + I_{22}
\end{align*}
and treat the local piece first. Since $v \varphi^2 \in H^{1,2}$, we can use the Fourier transform to justify
\begin{align*}
 (-\Delta)^{(1-\alpha)/2}(v \varphi^2) = - I_\alpha R^* \nabla(v \varphi^2)
\end{align*}
and as the Fourier symbol of $I_\alpha$ is $\RR$-valued, we can rewrite
\begin{align*}
I_2 = - \int_{\RR^n}  R I_\alpha(\mathbbm{1}_{4B}f) \cdot \cl{\nabla(v \varphi^2)}.
\end{align*}
From $f \in L^{2_*}$ and the above-mentioned boundedness properties of $R$ and $I_\alpha$ we can infer $R I_\alpha(\mathbbm{1}_{4B}f) \in L^2$ with norm controlled by $\|f\|_{L^{2_*}(4B)}$. Hence, $I_2$ is of the exact same nature as $I_1$ and we obtain a desirable bound by the same reasoning as before upon replacing $F$ with $R I_\alpha(\mathbbm{1}_{4B}f)$.

Finally, we use the integral representation \eqref{eq:fractional_Laplacian} for $(-\Delta)^{(1-\alpha)/2}(v\varphi^2)(x)$ to treat the global piece $I_{22}$. Since $v \varphi^2$ is supported in $2B$, there is no issue of convergence for $x \in (4B)^c$ and we get
\begin{align*}
 |(-\Delta)^{(1-\alpha)/2}(v\varphi^2)(x)| \leq c_{n,\alpha} \int_{2B} \frac{|v(y)|}{|x-y|^{n+1-\alpha}} \, dy.
\end{align*}
Splitting the integral in $x \in (4B)^c$ into dyadic annuli, we therefore obtain
\begin{align*}
 |I_{22}| 
 &\lesssim \sum_{k=2}^\infty  2^{-k(n+1-\alpha)}  \bigg(\int_{2B} |v(y)| \, dy \bigg)\bigg(\int_{2^k B} |f(x)| \, dx \bigg) \\
 &\leq  \int_{2B} |v|^2 + \bigg(\sum_{k=2}^\infty  2^{-k(n+1-\alpha)} \int_{2^k B} |f| \bigg)^2,
\end{align*}
where the second step follows from the elementary inequality $XY \leq X^2 + Y^2$ and H\"older's inequality on $X = \int_{2B} |v|$. A final application of Sobolev-Poincar\'e inequality to $v$ completes the proof of \eqref{eq:RH_RieszGradient_Goal}.
\end{proof}

\subsection{Reverse H\"older estimates incorporating the fractional Laplacian}

As our next step, we look for similar reverse H\"older estimates that incorporate $(-\Delta)^{\alpha/2}u$. Instead of looking separately at the fractional Laplacian, we shall incorporate this term into the estimate for the Riesz fractional gradient and obtain a reverse H\"older estimate for the sum $|D^\alpha u| + |(-\Delta)^{\alpha/2}u|$. 

\begin{proposition}
\label{prop:RH_Fractional_Laplacian}.
Let $u \in H^{\alpha,2}$ be a weak solution to \eqref{eq:fractional_equation}, where $f \in L^{2_*}$ and $F \in L^2$. Let $\rho \in (2_{*,1},2)$. Then for all balls $B = B(x,r) \subset \RR^n$,
\begin{align*}
 \bigg(\fint_B (|D^\alpha u|+|(-\Delta)^{\alpha/2}u|)^2 \bigg)^{1/2}
&\lesssim \bigg(\fint_{4B} |D^\alpha u|^\rho \bigg)^{1/\rho} + \bigg(\int_{2B} |(-\Delta)^{\alpha/2} u| \bigg) \\
&\quad+ \bigg(\fint_{4B} |F|^2 \bigg)^{1/2} + \bigg(\sum_{k=0}^\infty 2^{-k} \fint_{2^kB} |D^\alpha u| \bigg) \\
&\quad+ r^\alpha \bigg(  \sum_{k=0}^\infty  2^{-k(1-\alpha)} \fint_{2^k B} |f|^{2_{*,\alpha}}\bigg)^{1/2_{*,\alpha}},
\end{align*}
with an implicit constant depending on $n, \alpha, \lambda, \rho$.
\end{proposition}

This will follow at once from the preceding proposition and the following real variable lemma that has nothing to do with solutions to \eqref{eq:fractional_equation}. It does, however, illustrate how tails naturally enter the scene when changing the quantity to be controlled via a Riesz transform and we suggest that this phenomenon is natural also for more general Calder\'{o}n--Zygmund operators.

\begin{lemma}
\label{lem:localRiesz}
Let $\alpha \in (0,1)$ and $u \in H^{\alpha,2}$. There is a constant $C=C(n,\alpha)$ such that for any ball $B \subset \RR^n$,
\begin{align*}
 \bigg(\fint_{B} |(-\Delta)^{\alpha/2}u|^2\bigg)^{1/2} \leq C\bigg( \int_{2B} |D^\alpha u|^2 + \int_{B} |(-\Delta)^{\alpha/2}u| + \sum_{k=1}^\infty 2^{-k} \fint_{2^k B} |D^\alpha u| \bigg).
\end{align*}
\end{lemma}

\begin{proof}
By a scaling and translation argument it suffices again to argue for $B=B(0,1)$. Since $R^*R f= -f$ for every $f \in L^2$, we can split the square of the left-hand side as
\begin{align*}
 \int_B |(-\Delta)^{\alpha/2}u|^2
 &\leq 2\int_B |R^*(\mathbbm{1}_{2B}R(-\Delta)^{\alpha/2}u)|^2  + 2\int_B |R^*(\mathbbm{1}_{(2B)^c}R(-\Delta)^{\alpha/2}u)|^2.
\end{align*}
Using the relation $R (-\Delta)^{\alpha/2}u = D^\alpha u$ from \eqref{eq:relations_fractional_derivatives} as well as the $L^2$ boundedness of $R^*$ on the first integral, we find
\begin{align}
\label{eq:localRiesz1}
\int_B |(-\Delta)^{\alpha/2}u|^2 \lesssim \int_{2B} |D^\alpha u|^2 + \int_{B} |R^*(\mathbbm{1}_{(2B)^c}D^\alpha u)|^2
\end{align}
and it remains to control the right-most integral. In fact, we shall establish a pointwise bound for the integrand, which will yield the conclusion since $B$ is normalized. 

For brevity set $w:=D^\alpha u \in L^2$. Given $x \in B$, there is no issue of convergence with the integral representation \eqref{eq:RieszTransforms} for $R^*(\mathbbm{1}_{(2B)^c}w)(x)$ and we have
\begin{align*}
 R^*(\mathbbm{1}_{(2B)^c}w)(x) = c_n \int_{|y|>2} \frac{(x-y) \cdot w(y)}{|x-y|^{n+1}} \, dy.
\end{align*}
We let now $z \in B$ and $\eps \in (0,1)$ be free parameters to be specified later on, and introduce $w_0 := \mathbbm{1}_{2B} w \in L^1 \cap L^2$. Since $B(z,\eps) \subset 2B$, we have
\begin{align*}
 c_n^{-1} R^*(\mathbbm{1}_{(2B)^c}w)(x)
&= \int_{|y|>2} \Big( \frac{x-y}{|x-y|^{n+1}} - \frac{z-y}{|z-y|^{n+1}} \Big) \cdot w(y) \, dy \\
&\quad+ \int_{|z-y|> \eps} \frac{(z-y) \cdot w(y)}{|z-y|^{n+1}} \, dy 
-\int_{|z-y|> \eps} \frac{(z-y) \cdot w_0(y)}{|z-y|^{n+1}} \, dy. 
\end{align*}
The mean value theorem allows us to control the size of the kernel in the first integral on the right by $|y|^{-n-1}$. So, letting $\eps \to 0$ we obtain from \eqref{eq:RieszTransforms} for almost every $z \in B$ the bound
\begin{align*}
 |R^*(\mathbbm{1}_{(2B)^c}w)(x)| \lesssim \int_{|y|>2} \frac{|w(y)|}{|y|^{n+1}} \, dy + |R^*w(z)| + |R^*w_0(z)|,
\end{align*}
with an implicit constant depending on $\alpha$ and $n$. It remains to pick $z \in B$ correctly. Tchebychev's inequality entails
\begin{align*}
 \bigg| \bigg\{ z \in B: |R^*w(z)| \geq 4 \fint_B |R^*w| \bigg\} \bigg| \leq \frac{|B|}{4}
\end{align*}
and likewise the weak-$(1,1)$ bound for $R^*$ (with constant $C=C(n)$ say, see Theorem~II.4 in \cite{Stein}) guarantees
\begin{align*}
 \bigg| \bigg\{ z \in B: |R^*w_0(z)| \geq \frac{4 C}{|B|} \int_{\RR^n} |w_0| \bigg\} \bigg| \leq \frac{|B|}{4}.
\end{align*}
Hence, the set of $z \in B$ violating both conditions simultaneously has measure at least $|B|/2$ and we pick any such $z$. In conclusion, we have obtained for almost every $x \in B$ the pointwise bound
\begin{align*}
 |R^*(\mathbbm{1}_{(2B)^c}w)(x)| \lesssim \int_{|y|>2} \frac{|w(y)|}{|y|^{n+1}} \, dy + \int_{B} |R^*w| + \int_{\RR^n}|w_0|.
\end{align*}
At this stage we recall $w_0 = \mathbbm{1}_{2B} w$, $w = D^\alpha u$,  and thus we obtain from \eqref{eq:relations_fractional_derivatives} that $R^*w = R^* R (-\Delta)^{\alpha/2} u = - (-\Delta)^{\alpha/2} u$. Splitting the integral in $|y|>2$ into dyadic annuli we eventually find
\begin{align*}
 |R^*(\mathbbm{1}_{(2B)^c}w)(x)| \lesssim \int_B|(-\Delta)^{\alpha/2} u| + \sum_{k=1}^\infty 2^{-k} \fint_{2^k B} |D^\alpha u|. 
\end{align*}
Integrating both sides in $x \in B$ completes the ongoing estimate of the right-most integral in \eqref{eq:localRiesz1}.
\end{proof}

Now, we easily obtain a 

\begin{proof}[Proof of Proposition~\ref{prop:RH_Fractional_Laplacian}]
Proposition~\ref{prop:RH_RieszGradient} already controls the $L^2$ average of $D^\alpha u$ on $B$ by the desired right-hand side. In order to control the $L^2$ average of $(-\Delta)^{\alpha/2}u$, we first apply Lemma~\ref{lem:localRiesz} and then use Proposition~\ref{prop:RH_RieszGradient} again on the ball $2B$.
\end{proof}

\subsection{Proof of Theorem~\ref{thm:fractional}}

With Proposition~\ref{prop:RH_Fractional_Laplacian} at hand, we can apply our (non-local) Gehring lemmas to obtain improvement of the integrability of $|D^\alpha u| + |(-\Delta)^{\alpha} u|$. 

We begin with statement (1), where the right-hand side of the equations exhibits higher local integrability. Let us take any $\rho \in (1,2)$ to which Proposition~\ref{prop:RH_Fractional_Laplacian} applies, the precise value of which does neither play a role for the argument, nor the conclusion. We introduce the tail functional
\begin{align*}
 a_h(B) := \sum_{k=0}^\infty 2^{-k(1-\alpha)} \fint_{2^k B} h \, dx,
\end{align*}
for $h \geq 0$ locally integrable and $B \subset \RR^n$ a ball.

We put $v:= |(- \Delta)^{\alpha} u| + |(-\Delta)^{\alpha/2}u|$ so that Proposition~\ref{prop:RH_Fractional_Laplacian} entails for every ball $B = B(x,R)$,
\begin{align*}
 \bigg(\fint_B v^2 \bigg)^{1/2}
\lesssim (a_{v^{\,\rho}}(B))^{1/\rho} + (a_{|F|^2}(B))^{1/2} + R^\alpha (a_{|f|^{2_*}}(B))^{1/2_*},
\end{align*}
with an implicit constant depending on $n, \alpha, \lambda, \rho$. Note that here we have been very generous by using H\"older's inequality to unite all quantities containing either $(- \Delta)^{\alpha} u$ or $D^\alpha u$ in one single tail for $v$ with $L^{\, \rho}$-averages and reducing the decay of the geometric series meeting those averages. In order to bring this estimate in the form of Theorem~\ref{GeneralGehering.cor}, we introduce $\widetilde{v} := v^{\, \rho}$ and similarly $\widetilde{F}: = |F|^{\, \rho}$ and $\widetilde{f}:= |f|^{\, \rho}$. In terms of $\widetilde{v}, \widetilde{F}, \widetilde{f}$ the previous bound reads
\begin{align*}
 \bigg(\fint_B {\widetilde{v}}^{\, 2/\rho} \bigg)^{\rho/2}
\lesssim a_{\,\widetilde{v}}(B) + (a_{{\widetilde{F}}^{\, 2/\rho}}(B))^{\, \rho/2} + R^{\alpha \rho} (a_{{\widetilde{f}}^{\, 2_*/\rho}}(B))^{\, \rho/2_*}.
\end{align*}
Now that the exponent of $\widetilde{v}$ on the right-hand side is $1$, the claim follows from Theorem~\ref{GeneralGehering.cor} after checking the numerology. The parameters in that theorem are 
\begin{align*}
 (D,\beta,q,s) := \big(n,\alpha \rho,\tfrac{2}{\rho}, \tfrac{2_{*}}{\rho} \big)
\end{align*}
and so the conditions $0<s<q$, $q>1$ and $\beta \geq D(1/s - 1/q)$ are satisfied. (Note that in fact $D(1/s - 1/q) = \alpha \rho = \beta$ and that $s<1$). Hence, Theorem~\ref{GeneralGehering.cor} gives us local higher integrability for $\widetilde{v}$ with exponent larger than $q$ through the quantitative bound \eqref{GGeq2cor.eq}, provided $\widetilde{F}$ and $\widetilde{f}$ are globally integrable with exponents $q$ and $s$ and locally integrable to some higher exponents, respectively. By definition, this precisely means $F \in L^{2} \cap L_{loc}^{p}$ and $f \in L^{2_{\ast}} \cap L_{loc}^{p_{\ast}} $ and for some $p>2$, which is our assumption. 

Of course we can write the resulting estimate again in terms of the original functions: With $v :=  |D^\alpha u| + |(-\Delta)^{\alpha/2}|$ we get for all sufficiently small $\epsilon_0 = \epsilon_0(\lambda, n, \alpha, p)>0$ on all balls $B = B(x,R)$,
\begin{align}
\label{eq:RHbound_solution}
\begin{split}
\left( \fint_{B}  v^{2+\epsilon_0}\right)^{\tfrac{1}{2+\epsilon_0}} 
&\lesssim  \sum_{k=0}^{\infty} 2^{-k(1-\alpha) } \fint_{2^{k} B} v^{2} 
+ \left( \sum_{k=1}^{\infty} 2^{-k (1-\alpha)} \fint_{2^{k} B } |F|^{2} \right)^{1/2} \\
&\quad + R^{\alpha} \left(\sum_{k = 0}^\infty 2^{-k(1-\alpha)} \fint_{2^k B} |f|^{2_*} \right)^{1/2_{*}} \\
&\quad + \left( \fint_{2B } |F|^p \right)^{\tfrac{1}{p} }
+r^{\alpha}\left( \fint_{2 B } |f|^{p_*} \right)^{\tfrac{1}{p_*}},
\end{split}
\end{align}
where the right-hand side is finite due to our assumptions $F \in L^{2} \cap L_{loc}^{p}$ and $f \in L^{2_{\ast}} \cap L_{loc}^{p_{\ast}} $.

The global integrability stated in part (2) of the theorem follows by replacing Theorem \ref{GeneralGehering.cor} with Theorem \ref{thm:global1} in the proof above. \hfill $\square$

\end{document}